\documentclass[11pt,a4paper]{article}
\usepackage{amsfonts}
\usepackage{mathrsfs}
\usepackage{multirow}
\usepackage{color}
\usepackage{graphics}
\usepackage{subfigure}
\usepackage{algorithm}
\usepackage{algorithmic}
\usepackage{epsfig}
\usepackage{epstopdf}
\usepackage[T1]{fontenc}
\usepackage{geometry}
\usepackage{amsbsy,amsmath,latexsym,amsfonts, epsfig, color, authblk, amssymb, graphics, bm}
\usepackage{epsf,slidesec,epic,eepic}
\usepackage{fancybox}
\usepackage{fancyhdr}
\usepackage{setspace}
\usepackage{nccmath}
\usepackage{cases}
\usepackage{stmaryrd}
\usepackage[colorlinks, citecolor=blue]{hyperref}

\newtheorem{theorem}{Theorem}[section]
\newtheorem{lemma}{Lemma}[section]
\newtheorem{definition}{Definition}[section]

\newtheorem{proposition}{Proposition}[section]
\newtheorem{corollary}{Corollary}[section]
\newtheorem{assumption}{Assumption}
\newtheorem{remark}{Remark}[section]

\newenvironment{proof}{{\noindent \bf Proof:}}{\hfill$\Box$\medskip}

\definecolor{lred}{rgb}{1,0.8,0.8}
\definecolor{lblue}{rgb}{0.8,0.8,1}
\definecolor{dred}{rgb}{0.6,0,0}
\definecolor{dblue}{rgb}{0,0,0.5}
\definecolor{dgreen}{rgb}{0,0.5,0.5}

\title{Convergence analysis of inexact MBA method for constrained upper-$\mathcal{C}^2$ optimization problems}
\author{
	Ruyu Liu\footnote{(\href{mailto:maruyuliu@mail.scut.edu.cn}{maruyuliu@mail.scut.edu.cn}) School of Mathematics, South China University of Technology.}\ \ and\ \
	Shaohua Pan\footnote{(\href{mailto:shhpan@scut.edu.cn}{shhpan@scut.edu.cn}) School of Mathematics, South China University of Technology, Guangzhou.}
}
\date{}

\begin{document}

\maketitle

\begin{abstract}
This paper concerns a class of constrained optimization problems in which, the objective and constraint functions are both upper-$\mathcal{C}^2$. For such nonconvex and nonsmooth optimization problems, we develop an inexact moving balls approximation (MBA) method by a workable inexactness criterion for the solving of subproblems. By leveraging a global error bound for the strongly convex program associated with parametric optimization problems, we establish the full convergence of the iterate sequence under the partial bounded multiplier property (BMP) and the Kurdyka-{\L}ojasiewicz (KL) property of the constructed potential function, and achieve the local convergence rate of the iterate and objective value sequences if the potential function satisfies the KL property of exponent $q\in[1/2,1)$. A verifiable condition is also provided to check whether the potential function satisfies the KL property of exponent $q\in[1/2,1)$ at the given critical point. To the best of our knowledge, this is the first implementable inexact MBA method with a full convergence certificate for the constrained nonconvex and nonsmooth optimization problem. 
\end{abstract}

\section{Introduction}\label{sec1}

Constrained optimization problems are a class of tough composite optimization problems in which the outer nonsmooth function is an indicator of some closed convex set. They naturally arise when one attempts to seek a solution that minimizes an objective under some restrictions. On the other hand, nonconvex and nonsmooth functions abound with lower-$\mathcal{C}^2$ functions \cite{Clarke95,Bernard05}; for example, locally Lipschitz prox-regular functions are the common ones, so their negatives (upper-$\mathcal{C}^2$ functions) often appear in the difference programming problems \cite{Le2005dc,Le2018dc,Artacho25,Zeng25}. We are interested in developing an efficient algorithm for the constrained upper-$\mathcal{C}^2$ problem
\begin{equation}\label{prob}	\min_{x\in\mathbb{R}^n}F(x):=g_0(x)+\delta_{\mathbb{R}_{-}^m}(g(x))+\phi(x),
\end{equation}
where $\delta_{\mathbb{R}_{-}^m}(\cdot)$ is the indicator of the non-positive orthant $\mathbb{R}_{-}^m$, $g:=(g_1,\ldots,g_m)^{\top}$ is a mapping, and the functions $g_i$ for $i=0,1,\ldots,m$ and $\phi$ satisfy Assumption \ref{ass1} below. It is worth pointing out that the upper-$\mathcal{C}^2$ min-max problem $\min_{x\in\mathbb{R}^n}\max_{1\le i\le m}g_i(x)$ can also be reformulated as the form of \eqref{prob} by introducing an additional variable.
\begin{assumption}\label{ass1}
{\bf(i)} $g_0\!:\mathbb{R}^n\to\overline{\mathbb{R}}:=(-\infty,\infty]$ is locally Lipschitz and upper-$\mathcal{C}^2$ at every point of an open convex set $\mathcal{O}\supset\Gamma:=g^{-1}(\mathbb{R}_{-}^m)\ne\emptyset$, and  $g_i\!:\mathbb{R}^n\to\mathbb{R}$ for $i=1,\ldots,m$ are locally Lipschitz and upper-$\mathcal{C}^2$ at any point of $\mathbb{R}^n$; 

\noindent
{\bf(ii)} $\phi\!:\mathbb{R}^n\to\mathbb{R}$ is a convex function, and $F$ is bounded from below on the set $\Gamma$. 
\end{assumption}
 
Many methods have been developed for constrained optimization problems; for example, sequential quadratic programming (SQP) methods are known to be a class of popular ones for the classical nonlinear programming problems (NLPs), but they cannot be directly applied to cope with constrained nonsmooth problems like \eqref{prob}. Recently, some efforts have been made that incorporate certain key ideas from the classical algorithms to develop an iterative framework tailored for constrained nonsmooth problems. For the smooth equality and box constrained optimization problem with an upper-$\mathcal{C}^2$ objective function, Wang and Petra \cite{WangPetra23-siopt} proposed a nonsmooth SQP method with the proximal parameters updated by a trust-region strategy. Under the LICQ assumption on every accumulation point, they achieved the subsequence convergence of the iterate sequence. Mordukhovich et al. \cite{Zeng25} considered \eqref{prob} with $g_i=f_i\!-\!h_i$ for $i=0,1$ and $\phi=\delta_{X}$, where $f_i:\mathbb{R}^n\to\mathbb{R}$ are assumed to be $\mathcal{C}^{1+}$ (continuously differentiable with locally Lipschitzian gradients) and $h_i:\mathbb{R}^n\to\overline{\mathbb{R}}$ are locally Lipschitz and prox-regular over the closed convex set $X$. They proposed a novel variant of the extended sequential quadratic method introduced in \cite{Auslender13} for NLPs, and established the full convergence of the iterate sequence under the boundedness of penalty parameters and the KL property of the constructed potential function. 

Compared with the SQP methods, the moving ball approximation (MBA) method is a new type one, which was first proposed by Auslender et al. \cite{Auslender10} for NLPs. Different from the standard SQP method, it generates a sequence of feasible iterates from a feasible starting point. The authors of \cite{Auslender10} proved that under the MFCQ, any cluster point of the iterate sequence is a stationary point, and the whole sequence converges to a minimizer under the additional convexity. Later, for semialgebraic NLPs, Bolte and Pauwels \cite{Bolte16} established the full convergence and convergence rate of the iterate sequence for the MBA method under the MFCQ, thereby enhancing the convergence results of \cite{Auslender10}. Recently, Yu et al. \cite{YuPongLv21} studied a line-search variant of the MBA method for problem \eqref{prob} with $L$-smooth $g$. Under the MFCQ and the assumption that $g_0$ is $\mathcal{C}^{1+}$ on an open set containing the stationary point set of \eqref{prob}, they achieved the full convergence of the iterate sequence if the constructed potential function has the KL property, and the local convergence rate if the KL property is of exponent $q\in[1/2,1)$.  

The above MBA methods all require the subproblems to be solved exactly, which is unrealistic especially in large-scale applications due to computation error and cost. To overcome the issue, for problem \eqref{prob} with $\phi\equiv 0$ but every $g_i$ being a sum of an $L$-smooth function and a continuous convex function, Boob et al. \cite{Boob24} developed a level constrained proximal gradient (LCPG) method to allow the subproblems to be computed inexactly. They put more emphasis on the complexity analysis of LCPG under the MFCQ for stochastic and deterministic setting, though the subsequential convergence of the iterate sequence was also proved. Their inexactness criterion (see \cite[Definition 7]{Boob24}) involves the exact primal-dual optimal solutions of subproblems and the optimal values of subproblems, and now it is unclear whether such an inexactness criterion is workable or not in practical computation. For problem \eqref{prob} with all $g_i$ being $L$-smooth, Nabou and Necoara \cite{Nabou24} developed an inexact moving Taylor approximation (MTA) method with an implementable inexactness criterion, which reduces to the MBA method when the involved $p$ and $q$ both take $1$. Under the MFCQ, they established the full convergence of the iterate sequence if the constructed proximal Lagrange potential function has the KL property, and the local convergence rate if the potential function has the KL property of exponent $q\in[1/2,1)$. 

The convergence analysis in \cite{Nabou24} for the inexact MBA method depends heavily on the fact that an $L$-smooth function is globally upper-$\mathcal{C}^2$ as well as lower-$\mathcal{C}^2$. Obviously, it is inapplicable to problem \eqref{prob} where all $g_i$ are only assumed to be locally upper-$\mathcal{C}^2$. Then, there is still a lack of a practical MBA method with convergence guarantee for solving \eqref{prob}. On the other hand, the existing asymptotic convergence results on the MBA method are all obtained for \eqref{prob} with $L$-smooth $g_i$ under the MFCQ. The robustness of MFCQ ensures the uniform boundedness of the Lagrange multipliers of subproblems, which greatly reduces the difficulty caused by the multipliers of subproblems. For NLPs, the past two decades have witnessed that the weaker metric subregularity constraint qualification (MSCQ) was used to derive optimality conditions and study stability of constraint system (see, e.g., \cite{Gfrerer17,Gfrerer19,Bai19}). Then, it is significant to study the convergence of the MBA method under weaker conditions.  
\subsection{Main contributions}\label{sec1.1}

This work aims to develop an inexact MBA method with a full convergence certificate for \eqref{prob}. Its main contributions are as follows.
\begin{itemize}
\item First, we formulate a local MBA model at the current iterate by seeking a tight upper estimation for the upper-$\mathcal{C}^2$ modulus of $g_i\ (i=0,1,\ldots,m)$ with the line search strategy, and propose an implementable inexactness criterion involving the approximate stationarity, complementarity, and feasibility for the computation of the local MBA model. This inexactness criterion is a little weaker than the one proposed in \cite{Nabou24} since the complementarity violation is only related to that of the feasible constraints. 

\item Second, a potential function $\Phi_{\widetilde{c}}$ is constructed by fully leveraging the structure of upper-$\mathcal{C}^2$ functions on a compact convex set; see \eqref{Phic-def} for its expression. Different from the value function adopted in \cite{Bolte16}, our $\Phi_{\widetilde{c}}$ takes into account the inexact solving of subproblems. Unlike the Lagrange potential function in \cite{Nabou24}, $\Phi_{\widetilde{c}}$ does not involve any inexact Lagrange multipliers of subproblems, which dispenses with considering their boundedness and makes it possible to weaken the MFCQ. Notably, the partial bounded multiplier property (BMP) on the set of cluster points along with a global error bound, derived for the strongly convex programs associated with a parametric optimization problem arising from the approximation to \eqref{prob}, allows us to establish the crucial error bound for the inexact solutions of subproblems (see Proposition \ref{prop-eboundk}), and consequently, the bridge between $\Phi_{\widetilde{c}}$ and the objective function $F$ and the relative inexact optimality of an augmented iterate sequence. Then, the whole iterate sequence is proved to converge to a stationary point under the KL property of $\Phi_{\widetilde{c}}$. As discussed in Remark \ref{remark-ass4}, the partial BMP is weaker than the MFCQ for \eqref{prob} and the independent partial constant rank constraint qualification (CRCQ). Thus, for the first time, we achieve the full convergence of the iterate sequence for the MBA method without the MFCQ for \eqref{prob}.

\item Third, a verifiable criterion is provided for checking whether the potential function $\Phi_{\widetilde{c}}$ satisfies the KL property of exponent $q\in[1/2,1)$ on the interested critical point, so that the convergence of the iterate sequence admits a linear or sublinear rate. This criterion is weaker than the one obtained in \cite[Theorem 3.2]{LiPong18} for identifying the KL property of exponent $q\in[0,1)$ of general composite functions, though it is stronger than the MFCQ for \eqref{prob}.   
\end{itemize}

We also discuss the implementation of the proposed inexact MBA method by applying the proximal gradient method with line-search (PGls) to solve the dual of subproblems. Numerical comparisons with the software MOSEK on several classes of test problems are included in the archive version of this paper, which indicate that our iMBA-PGls can produce better solutions within less time for large-scale problems. 
\subsection{Notation}\label{sec1.2}

Throughout this paper, $\mathcal{I}$ denotes the identity mapping from $\mathbb{R}^n$ to $\mathbb{R}^n$, $\|\cdot\|$ (resp. $\|\cdot\|_1$ and $\|\cdot\|_{\infty}$) signifies the $\ell_2$-norm (resp. $\ell_1$-norm and $\ell_{\infty}$-norm), and $\mathbb{B}(x,\delta)$ denotes the closed ball centered at $x$ with radius $\delta>0$, whose interior is written as $\mathbb{B}^{\circ}(x,\delta)$. For an integer $k\ge 1$, write $[k]_{+}:=\{0,1,\ldots,k\}$ and $[k]:=\{1,\ldots,k\}$. For a linear mapping $\mathcal{Q}\!:\mathbb{R}^n\to\mathbb{R}^n$, $\mathcal{Q}^*$ denotes its adjoint, $\mathcal{Q}\succeq 0$ means that $\mathcal{Q}$ is positive semidefinite (PSD), i.e., $\mathcal{Q}=\mathcal{Q}^*$ and $\langle x,\mathcal{Q}x\rangle\ge 0$ for all $x\in\mathbb{R}^n$, and  $\|\mathcal{Q}\|\!:=\max_{\|x\|=1}\|\mathcal{Q}x\|$ denotes its spectral norm. For a matrix $V\in\mathbb{R}^{n\times m}$, $V_{i}$ represents the $i$th column of $V$. For a vector $x\in\mathbb{R}^n$, $x\in[a,b]$ for $a,b\in\mathbb{R}$ with $a\le b$ means $a\le x_i\le b$ for all $i\in[n]$, and $x_{+}$ denotes the vector whose $i$th component is $\max\{0,x_i\}$. For a closed set $C\subset\mathbb{R}^n$, $\delta_{C}(\cdot)$ stands for the indicator function of $C$, i.e., $\delta_{C}(x)=0$ if $x\in C$; otherwise $\delta_{C}(x)=\infty$, and ${\rm dist}(x,C)$ denotes the distance on the $\ell_2$-norm from $x$ to $C$. For a proper lsc function $h\!:\mathbb{R}^n\to\overline{\mathbb{R}}$, $[a<h<b]$ with real numbers $a<b$ denotes the set $\{x\in\mathbb{R}^n\,|\,a<h(x)<b\}$, and $h^*(s):=\sup_{x\in\mathbb{R}^n}\{\langle x,s\rangle-h(x)\}$ means its conjugate function. For a mapping $H\!=(H_1,\ldots,H_m)^{\top}\!:\mathbb{R}^n\to\mathbb{R}^m$, if $H$ is differentiable at $\overline{x}$, write $\nabla\!H(\overline{x}):=[\mathcal{J}\!H(\overline{x})]^{\top}$ where $\mathcal{J}\!H(\overline{x})$ represents the Jacobian matrix of $H$ at $\overline{x}$. 

\section{Preliminaries}\label{sec2} 

We first introduce some knowledge on multifunctions related to this paper, and more details can be found in \cite{RW98,Mordu18book}. A multifunction  $\mathcal{F}\!:\mathbb{R}^n\rightrightarrows\mathbb{R}^m$ is said to be outer semicontinuous (osc) at $\overline{x}\in\mathbb{R}^n$ if $\limsup_{x\to\overline{x}}\mathcal{F}(x)\subset\mathcal{F}(\overline{x})$, and it is said to be locally bounded at $\overline{x}\in\mathbb{R}^n$ if for some neighborhood $\mathcal{V}$ of $\overline{x}$, the set $\mathcal{F}(\mathcal{V}):=\bigcup_{x\in\mathcal{V}}\mathcal{F}(x)$ is bounded. Denote by ${\rm gph}\,\mathcal{F}$ the graph of the mapping $\mathcal{F}$. 
\begin{definition}\label{def-subdiff} 
 (see \cite[Definition 8.3]{RW98}) Consider a function $h\!:\mathbb{R}^n\to\overline{\mathbb{R}}$ and a point $x\in{\rm dom}\,h$. The regular (or Fr\'echet) subdifferential of $h$ at $x$ is defined as
 \[
  \widehat{\partial}h(x):=\bigg\{v\in\mathbb{R}^n\ |\ \liminf_{x'\to x\atop x'\neq x}\frac{h(x')-h(x)-\langle v,x'-x\rangle}{\|x'-x\|}\ge 0\bigg\}, 
 \]
 and its basic (known as limiting or Morduhovich) subdifferential at $x$ is defined as 
 \[
  \partial h(x):=\Big\{v\in \mathbb{R}^n\ |\ \exists\, x^k\to x\ {\rm with}\ h(x^k)\to h(x), v^k\!\in \widehat{\partial}h(x^k)\ {\rm with}\ v^k \to v \Big\}. 
 \]
\end{definition}
\begin{remark}\label{remark-subdiff}
{\bf(a)} When $h$ is locally Lipschitz at $x$, the convex hull of $\partial h(x)$ is the Clarke subdifferential of $h$ at $x$, denoted by $\overline{\partial}h(x)$. By virtue of Assumption \ref{ass1}, at any $x\in\Gamma$,  $\overline{\partial}g_i(x)={\rm co}\,\partial g_i(x)$ for all $i\in[m]_{+}$. For convenience, in the sequel, we write
\[
 \overline{\partial} g(x):={\rm co}\,\partial g(x)\ {\rm with}\ \partial g(x):=\big\{V\in\mathbb{R}^{n\times m}\,|\,V_{i}\in\partial g_i(x)\ {\rm for\ all}\ i\in[m]\big\}. 
\]
\noindent
{\bf(b)} The mappings $\partial g_0\!:\mathbb{R}^n\rightrightarrows\mathbb{R}^n$ and $\partial g\!:\mathbb{R}^n\rightrightarrows\mathbb{R}^{n\times m}$ are osc at all $x\in\Gamma$, and they are locally bounded at such $x$ by \cite[Theorem 9.13 (d)]{RW98}. The mapping $\partial\phi\!:\mathbb{R}^n\rightrightarrows\mathbb{R}^n$ is osc and locally bounded at all $x\in\mathbb{R}^n$ by \cite[Theorem 9.13 (d)]{RW98}.
\end{remark}

When $h$ is an indicator function of a closed set $\Delta\subset\mathbb{R}^n$, its regular subdifferential at $x\in {\rm dom}\,h$ becomes the regular normal cone to $\Delta$ at $x$, denoted by $\widehat{\mathcal{N}}_{\Delta}(x)$, while its subdifferential at $x\in {\rm dom}\,h$ reduces to the (limiting) normal cone to $\Delta$ at $x$, denoted by $\mathcal{N}_{\Delta}(x)$. When $\Delta$ is convex, $\widehat{\mathcal{N}}_{\Delta}(x)=\mathcal{N}_{\Delta}(x)=\big\{v\in\mathbb{R}^n\ |\ \langle v,z-x\rangle\le0\ \ \forall z\in\Delta\big\}$.  
\subsection{The family of upper-$\mathcal{C}^2$ functions}\label{sec2.1}
The class of upper-$\mathcal{C}^2$ functions is a larger family of nonsmooth functions that satisfy an inequality similar to the one for the $L$-smooth function, preserving thus the same nice properties for optimization
(see, e.g., \cite[Theorem 5.1]{Clarke95} and \cite[Definition 10.29]{RW98}). Consider a function $h:\mathbb{R}^n\to\overline{\mathbb{R}}$ and a point  $\overline{x}\in{\rm dom}\,h$. Note that $h$ is upper-$\mathcal{C}^2$ at $\overline{x}$ iff $-h$ is lower-$\mathcal{C}^2$ at $\overline{x}$. From \cite[Theorem 10.33]{RW98} and its proof, $h$ is upper-$\mathcal{C}^2$ at $\overline{x}$ iff there exist $\rho\ge 0,\delta>0$, a compact set $\Delta\subset\mathbb{R}^p$ and some continuous $b:\Delta\to\mathbb{R}^n$ and $c:\Delta\to\mathbb{R}$ such that 
\[
 h(x)=\min_{s\in\Delta}\big\{(\rho/2)\|x\|^2-\langle b(s),x\rangle-c(s)\big\}\quad{\rm for\ all}\ x\in\mathbb{B}^{\circ}(\overline{x},\delta),
\] 
and in the sequel we call $\rho$ the upper-$\mathcal{C}^2$ constant of $h$ at $\overline{x}$. 
Next we present two important properties of upper-$\mathcal{C}^2$ functions, which are often used in Sections \ref{sec3}-\ref{sec4}. 
\begin{lemma}\label{lemma-major}
 Let $h:\mathbb{R}^{n}\to\overline{\mathbb{R}}$ be locally Lipschitz continuous and upper-$\mathcal{C}^2$ at every point of an open convex set $\mathcal{O}'\supset{\rm dom}\,h$. Then, the following two assertions hold. 
 \begin{itemize}
 \item[(i)] For any $x\in\mathcal{O}'$, there exist $\delta>0$ and $\rho\ge 0$ such that for all $x',x''\in\mathbb{B}^{\circ}(x,\delta)$ and $\zeta'\in\overline{\partial} h(x')$,
 $h(x'')\le h(x')+\langle \zeta',x''-x'\rangle+(\rho/{2})\|x''-x'\|^2$. 

 \item[(ii)] At any nonempty compact convex set $D\subset\mathcal{O}'$, $h$ can be expressed as the difference of a strongly convex quadratic function and a strongly convex function, so there exists a constant $\widetilde{\rho}>0$ such that for all $x,y\in D$ and $\xi\in\overline{\partial}h(x)$,
 \[
   h(y)\le h(x)+\langle\xi,y-x\rangle+(\widetilde{\rho}/2)\|y-x\|^2.
 \]
\end{itemize} 
\end{lemma}
\begin{proof}
Part (i) directly follows \cite[Proposition 2.2 (i)-(ii)]{Artacho25}. The rest focuses on the proof of part (ii). By \cite[Theorem 10.33]{RW98}, for each $y\in D$, there exists a neighborhood $\mathbb{B}^{\circ}(y,\delta_{y})$ such that $h$ is locally expressed as the difference of a quadratic function $\frac{\rho_{y}}{2}\|\cdot\|^2$ and a finite convex function $\widehat{h}_{y}$ on $\mathbb{R}^n$. Notice that $\bigcup_{y\in D}\mathbb{B}^{\circ}(y,\delta_{y})$ is an open covering of $D$. From the Heine-Borel theorem, there exist $y_1,\ldots,y_{p}\in D$ and $\delta_1>0,\ldots,\delta_{p}>0$ for some $p\in\mathbb{N}$ such that $D\subset\bigcup_{j=1}^p\mathbb{B}^{\circ}(y_{j},\delta_j)$, and on $\mathbb{B}^{\circ}(y_{j},\delta_j)$, $h=\frac{\rho_{y_j}}{2}\|\cdot\|^2-\widehat{h}_{y_j}$ where $\widehat{h}_{y_j}$ is a finite convex function on $\mathbb{R}^n$. Write $\rho:=\max_{j\in[p]}\rho_{y_j}$ and $\widehat{h}:=({\rho}/{2})\|\cdot\|^2-h$. Observe that $\widehat{h}:\mathbb{R}^n\to\mathbb{R}$ is convex on every $\mathbb{B}^{\circ}(y_{j},\delta_j)$ for $j\in[p]$. By \cite[Theorem 7.102]{Mordu22book}, $\widehat{h}$ is convex on $\bigcup_{j=1}^p\mathbb{B}^{\circ}(y_{j},\delta_j)$, so is on the set $D$. Thus, $h(x)=(\rho/2)\|x\|^2-\widehat{h}(x)$ for $x\in D$. Pick any $\gamma>0$. Then, it holds
\begin{equation*}
 h(x)=\frac{\rho\!+\!\gamma}{2}\|x\|^2-(\frac{\gamma}{2}\|x\|^2+\widehat{h}(x))\quad\forall x\in D.
\end{equation*}
The first part then follows, so $h$ is a $(\widetilde{\rho}/2)$-upper-$\mathcal{C}^2$ function on $D$ with $\widetilde{\rho}={\rho+\gamma}$. Thus, the desired inequality follows \cite[Proposition 2.2 (i)-(ii)]{Artacho25}. 
\end{proof}
\subsection{Constraint qualifications}\label{sec2.2}

We first introduce the (nonsmooth) MFCQ. 
\begin{definition}\label{GMFCQ}
 For every $x\in\Gamma$, write $I(x)\!:=\{i\in[m]\,|\,g_i(x)=0\}$. We say that the MFCQ holds at a feasible point $x$ of \eqref{prob} if for any $\zeta^i\in\partial g_i(x)$ with $i\in I(x)$, there exists a vector $d\in\mathbb{R}^n$ such that $\langle \zeta^i,d\rangle<0$ for all $i\in I(x)$. 
\end{definition}

It is not hard to check that the MFCQ at $x\in\Gamma$ in Definition \ref{GMFCQ} is equivalent to 
\begin{equation}\label{EMFCQ}
  0\in\partial\langle\lambda,g\rangle(x),\lambda\in\mathcal{N}_{\mathbb{R}_{-}^m}(g(x))\ \Longrightarrow\ \lambda=0, 
\end{equation}
and is weaker than the NNAMCQ introduced in \cite[Definition 2.5]{Zeng25} for $m=1$, where the limiting subdifferential $\partial\langle\lambda,g\rangle(x)$ is replaced by its Clarke one. Define 
\begin{equation}\label{mapFH}
\mathcal{F}_{\!g}(z):=g(z)-\mathbb{R}_{-}^m\quad{\rm for}\ z\in\mathbb{R}^n.
\end{equation}
By leveraging the Morduhovich rule in \cite[Theorem 9.40]{RW98} for identifying the metric regularity of a multifunction, one can verify that the MFCQ at $\overline{x}\in\Gamma$ implies the metric regularity of $\mathcal{F}_{\!g}$ at $(\overline{x},0)$, and they are equivalent if $g$ is continuously differentiable at $\overline{x}$ (see \cite[Section 2.3.3]{BS00}). Metric regularity of a multifunction is stronger than its (metric) subregularity. A multifunction $\mathcal{F}\!:\mathbb{R}^n\rightrightarrows\mathbb{R}^m$ is said to be subregular at $(\overline{x},\overline{y})\in{\rm gph}\,\mathcal{F}$ if there exist $\delta>0$ and $\kappa>0$ such that for all $x\in\mathbb{B}(\overline{x},\delta)$,
\[
 {\rm dist}(x,\mathcal{F}^{-1}(\overline{y}))\le\kappa\,{\rm dist}(\overline{y},\mathcal{F}(x)).
\]
If the mapping $\mathcal{F}_{\!g}$ is subregular at $(x,0)$ with $x\in\Gamma$ (now it is said that the MSCQ holds at $x\in\Gamma$ for \eqref{prob}), from \cite[Page 211]{Ioffe08-calm}, we get
\(
 \mathcal{N}_{\Gamma}(x)\subset{\textstyle\bigcup_{\zeta\in\mathcal{N}_{\mathbb{R}_{-}^m}(g(x))}}\,\partial\langle\zeta,g\rangle(x).
\)
When all components of $g$ are smooth and convex, from the proof for the sufficiency of \cite[Theorem 3.5]{Liwu97},   
the subregularity of $\mathcal{F}_{\!g}$ at $x\in\Gamma$ for $0$ is implied by the Abadie's CQ at $x\in\Gamma$ for problem \eqref{prob}, the weakest condition for the characterization of an optimal solution in terms of the Karush-Kuhn-Tucker (KKT) conditions. 

Let $x^*$ be a local optimal solution of \eqref{prob}. By  \cite[Theorem 10.1 \& Corollary 10.9]{RW98}, $0\in\partial g_0(x^*)+\partial\phi(x^*)+\mathcal{N}_{\Gamma}(x^*)$ where, if $\mathcal{F}_{\!g}$ is subregular at $(x^*,0)$, $\mathcal{N}_{\Gamma}(x^*)\subset\bigcup_{\zeta\in\mathcal{N}_{\mathbb{R}_{-}^m}(g(x^*))}\partial\langle\zeta,g\rangle(x^*)\subset\bigcup_{\zeta\in\mathcal{N}_{\mathbb{R}_{-}^m}(g(x^*))}\sum_{i=1}^m\zeta_i\partial g_i(x^*)$. This inspires us to introduce the following notion of stationary points for problem \eqref{prob}. 
\begin{definition}\label{spoint-def}
 A vector $x\in\Gamma$ is said to be a stationary point of \eqref{prob} if
 \[	
  0\in\partial g_0(x)+\partial\phi(x)+{\textstyle\bigcup_{\zeta\in\mathcal{N}_{\mathbb{R}_{-}^m}(g(x))}}\textstyle{\sum_{i=1}^m}\zeta_i\partial g_i(x).
\]
In the sequel, we denote by $\Gamma^*$ the set of stationary points of problem \eqref{prob}.
\end{definition}

When $g_0$ is the difference of a $\mathcal{C}^{1+}$ function and a finite convex function, and all $g_i$ for $i\in[m]$ are smooth, the above stationary point is stronger than the one in the literature on DC problems (see \cite[Definition 2.2]{YuPongLv21} or \cite[Theorem 4]{Le24} with $F_1(x)\equiv x$). 
\subsection{Partial BMP of parametric system}\label{sec2.3}

Consider the parametric problem
\begin{equation}\label{para-prob}
	\min_{x\in\mathbb{R}^n}\big\{\theta(u,x)+\phi(x)\ \ {\rm s.t.}\ \ H(u,x)\in\mathbb{R}_{-}^m\big\},
\end{equation} 
where $u$ is the parameter variable from a finite dimensional real vector space $\mathbb{U}$, and $\theta\!:\mathbb{U}\times\mathbb{R}^n\to\mathbb{R}$ and $H\!:\mathbb{U}\times\mathbb{R}^n\to\mathbb{R}^m$ are continuously differentiable functions. The stability of the parametric constraint system $H(u,x)\in\mathbb{R}_{-}^m$ was studied in \cite{Gfrerer17} under the partial BMP and MSCQ, which is applicable to more models since their validity is implied by the partial MFCQ or the independent partial CRCQ. Define
\[
  \mathcal{S}(u):=\big\{x\in\mathbb{R}^n\ |\  H(u,x)\in\mathbb{R}_{-}^m\big\}\quad{\rm for}\ u\in\mathbb{U}.
\]
Now we recall from \cite[Definition 3.1]{Gfrerer17} the formal definition of the partial BMP.     
\begin{definition}\label{def-BMP}
The constraint system of \eqref{para-prob} is said to satisfy the partial BMP w.r.t. $x$ at $(u^*,x^*)\in{\rm gph}\,\mathcal{S}$ if there exist $\kappa>0$ and a neighborhood $\mathcal{U}\times\mathcal{V}$ of $(u^*,x^*)$ such that for all $u\in\mathcal{U},x\in\mathcal{V}\cap \mathcal{S}(u)$ and $y\in\mathcal{N}_{\mathcal{S}(u)}(x)$, $\Lambda(u,x,y)\cap\kappa \|y\|\mathbb{B}\ne\emptyset$ with
\begin{equation}\label{Lambda-map}
 \Lambda(u,x,y)\!:=\big\{\lambda\in\mathcal{N}_{\mathbb{R}_{-}^m}(H(u,x))\ |\ \nabla_{\!x}H(u,x)\lambda=y\big\}.
\end{equation}
\end{definition}  

Recall that $x^*\in\mathbb{R}^n$ is a stationary point of \eqref{para-prob} associated with $u=u^*\in\mathbb{U}$ if there exists  $\lambda^*\!\in\mathcal{N}_{\mathbb{R}_{-}^m}(H(u^*,x^*))$ such that
$0\in\nabla_{\!x} \theta(u^*,x^*)+\partial\phi(x^*)+\nabla_{\!x}H(u^*,x^*)\lambda^*$. We define the Lagrange multiplier set of \eqref{para-prob} associated with $(u,x)$ by
\begin{equation}\label{para-multiplier}
 \mathcal{M}(u,x)\!:=\!\Big\{\lambda\in\mathcal{N}_{\mathbb{R}_{-}^m}(H(u,x))\ |\ 0\in\nabla\!_x\theta(u,x)\!+\!\partial\phi(x)+\nabla_{\!x}H(u,x)\lambda\Big\}.
\end{equation}
Clearly, $\mathcal{M}(u^*,x^*)\ne\emptyset$ if $x^*$ is a stationary point of \eqref{para-prob} associated with $u=u^*$, and in this case comparing with the expression of $\Lambda$ in \eqref{Lambda-map} we conclude that
\begin{equation}\label{Mset-Lambda}
\textstyle{\bigcup_{v^*\in\partial\phi(x^*)}}\,\Lambda(u^*,x^*,-\nabla_{\!x}\theta(u^*,x^*)-v^*)\subset\mathcal{M}(u^*,x^*).
\end{equation}

Next, for a strongly convex problem \eqref{para-prob} associated with some $\overline{u}\in\mathbb{U}$, we establish a global error bound for the distance of any $x\in\mathbb{R}^n$ from its unique solution in terms of the violation of KKT conditions, which will be used in Section \ref{sec4} to prove Proposition \ref{prop-eboundk}. Its proof is similar to that of \cite[Theorem 2.2]{Mangasarian88} but needs a weaker CQ. 
\begin{proposition}\label{prop-ebound} 
 Fix any $\overline{u}\in\mathbb{U}$. Suppose that $\theta(\overline{u},\cdot)$ is strongly convex with modulus $\rho(\overline{u})$ and all components of $H(\overline{u},\cdot)$ are convex with $\mathcal{S}(\overline{u})\ne\emptyset$. Let $\overline{x}$ be the unique optimal solution of \eqref{para-prob} associated with $\overline{u}\in\mathbb{U}$. If $\mathcal{H}_{\overline{u}}(\cdot):=H(\overline{u},\cdot)-\mathbb{R}_{-}^m$ is subregular at $(\overline{x},0)$, then for any $\overline{\lambda}\in\mathcal{N}_{\mathbb{R}_{-}^m}(H(\overline{u},\overline{x}))$, $(x,\lambda)\in\mathbb{R}^n\times\mathbb{R}_+^m$ and $v\in\partial\phi(x)$, 
 \begin{align*}
  \rho(\overline{u})\|x-\overline{x}\|^2&\le\interleave x\!-\!\overline{x}\interleave\interleave\nabla_{\!x}\theta(\overline{u},x)\!+\!v\!+\!\nabla_{\!x}H(\overline{u},x)\lambda\interleave_{*}\\
  &\quad +\interleave\overline{\lambda}\interleave\interleave[H(\overline{u},x)]_{+}\interleave_{*}-\langle\lambda,H(\overline{u},x)\rangle,
 \end{align*} 
 where $\interleave\cdot\interleave$ represents an arbitrary norm in $\mathbb{R}^{p}$ and $\interleave\cdot\interleave_{*}$ is the dual norm of $\interleave\cdot\interleave$.
\end{proposition}
\begin{proof}
Since $\mathcal{S}(\overline{u})$ is a nonempty closed convex set and $\theta(\overline{u},\cdot)$ is strongly convex, the problem \eqref{para-prob} associated with $\overline{u}$ has a unique optimal solution $\overline{x}$. By the first-order optimality condition, $0\in\nabla_{\!x}\theta(\overline{u},\overline{x})+\partial\phi(\overline{x})+\mathcal{N}_{\mathcal{S}(\overline{u})}(\overline{x})$. Also, the subregularity of $\mathcal{H}_{\overline{u}}$ at $(\overline{x},0)$ and the smoothness of $H(\overline{u},\cdot)$ implies 
$\mathcal{N}_{\mathcal{S}(\overline{u})}(\overline{x})=\nabla_{\!x} H(\overline{u},\overline{x})\mathcal{N}_{\mathbb{R}_{-}^m}(H(\overline{u},\overline{x}))$. Pick any $\overline{\lambda}\in\mathcal{N}_{\mathbb{R}_{-}^m}(H(\overline{u},\overline{x}))$. Then, there exists $\overline{v}\in\partial\phi(\overline{x})$ such that 
\begin{equation}\label{ineq-KKT}
\nabla_{\!x}\theta(\overline{u},\overline{x})+\overline{v}+\nabla_{\!x} H(\overline{u},\overline{x})\overline{\lambda}=0.
\end{equation}
 Since every $H_{i}(\overline{u},\cdot)$ for $i\in[m]$ is assumed to be convex, it holds that
 \begin{equation*}
  H(\overline{u},y)-H(\overline{u},z)\ge \mathcal{J}_xH(\overline{u},z)(y-z)\quad\forall y,z\in\mathbb{R}^n.
 \end{equation*}  
 Now fix any $(x,\lambda)\in\mathbb{R}^n\times\mathbb{R}^m_+$ and $v\in\partial\phi(x)$. Using $\lambda\in\mathbb{R}_{+}^m$ and $\overline{\lambda}\in\mathbb{R}_{+}^m$ and invoking the above inequality with $(y,z)=(x,\overline{x})$ and $(y,z)=(\overline{x},x)$, respectively, yields 
 \begin{equation}\label{Hconvex-ineq}
  \left\{\begin{array}{cl}
  \langle\overline{\lambda},H(\overline{u},x)-H(\overline{u},\overline{x})\rangle\ge \langle\nabla_{x} H(\overline{u},\overline{x})\overline{\lambda}, x-\overline{x}\rangle,\\
 \langle\lambda,H(\overline{u},\overline{x})-H(\overline{u},x)\rangle\ge \langle\nabla_{x} H(\overline{u},x)\lambda, \overline{x}-x\rangle.
 \end{array}\right.
 \end{equation}
 From the strong convexity of $\theta(\overline{u},\cdot)$ with modulus $\rho(\overline{u})$ and the convexity of $\phi$,
 \begin{align*}
  \rho(\overline{u})\|x-\overline{x}\|^2&\le\langle x-\overline{x},\nabla_{\!x}\theta(\overline{u},x)-\nabla_{\!x}\theta(\overline{u},\overline{x})\rangle+\langle x-\overline{x},v-\overline{v}\rangle\nonumber\\
  &=\langle x-\overline{x},\nabla_{\!x}\theta(\overline{u},x)+\nabla_{\!x} H(\overline{u},x)\lambda-\nabla_{\!x}\theta(\overline{u},\overline{x})-\nabla_{\!x} H(\overline{u},\overline{x})\overline{\lambda}\rangle\\
  &\quad\ +\langle x-\overline{x},\nabla_{\!x} H(\overline{u},\overline{x})\overline{\lambda}\rangle-\langle x-\overline{x},\nabla_{\!x} H(\overline{u},x)\lambda\rangle+\langle x-\overline{x},v-\overline{v}\rangle\nonumber\\
  &\stackrel{\eqref{Hconvex-ineq}}{\le}\langle x-\overline{x},\nabla_{\!x}\theta(\overline{u},x)+\nabla_{\!x} H(\overline{u},x)\lambda-\nabla_{\!x}\theta(\overline{u},\overline{x})-\nabla_{\!x} H(\overline{u},\overline{x})\overline{\lambda}\rangle\\
  &\quad +\langle\overline{\lambda},H(\overline{u},x)-H(\overline{u},\overline{x})\rangle+\langle\lambda,H(\overline{u},\overline{x})-H(\overline{u},x)\rangle+\langle x-\overline{x},v-\overline{v}\rangle\nonumber\\
  &\stackrel{\eqref{ineq-KKT}}{\le}\langle x-\overline{x},\nabla_{\!x}\theta(\overline{u},x)+v+\nabla_{\!x} H(\overline{u},x)\lambda\rangle+\langle\overline{\lambda},H(\overline{u},x)\rangle-\langle\lambda,H(\overline{u},x)\rangle\nonumber\\
  &\le \interleave x\!-\!\overline{x}\interleave\interleave\nabla_{\!x}\theta(\overline{u},x)+v+\nabla_{\!x}H(\overline{u},x)\lambda\interleave_{*}\nonumber\\
  &\quad+\interleave\overline{\lambda}\interleave\interleave[H(\overline{u},x)]_{+}\interleave_{*}-\langle\lambda,H(\overline{u},x)\rangle
 \end{align*}
 where the third inequality is also due to  $\overline{\lambda}\in\mathcal{N}_{\mathbb{R}_{-}^m}(H(\overline{u},\overline{x}))$ and $\lambda\in\mathbb{R}_+^m$. 
\end{proof}
\subsection{Kurdyka-{\L}ojasiewicz property}\label{sec2.4}

It is well recognized that the KL property plays a key role in achieving full convergence and convergence rate of descent methods for nonconvex and nonsmooth optimization after the seminal works \cite{Attouch09,Attouch13,Bolte14}.  
\begin{definition}\label{KL-def}
 For every $\eta>0$, denote $\Upsilon_{\!\eta}$ by the set of continuous concave functions $\varphi\!:[0,\eta)\to\mathbb{R}_{+}$ that are continuously differentiable in $(0,\eta)$ with $\varphi(0)=0$ and $\varphi'(s)>0$ for all $s\in(0,\eta)$. A proper function $h\!:\mathbb{R}^n\to\overline{\mathbb{R}}$ is said to have the KL property at a point $\overline{x}\in{\rm dom}\,\partial h$ if there exist $\delta>0,\eta\in(0,\infty]$ and $\varphi\in\Upsilon_{\!\eta}$ such that 
 \[
	\varphi'(h(x)\!-\!h(\overline{x})){\rm dist}(0,\partial h(x))\ge 1\quad\forall x\in\mathbb{B}(\overline{x},\delta)\cap[h(\overline{x})<h<h(\overline{x})+\eta],
 \]
 and it is said to have the KL property of exponent $q\in[0,1)$ at $\overline{x}$ if the above $\varphi$ can be chosen to be the function $\mathbb{R}_{+}\ni t\mapsto ct^{1-q}$ for some $c>0$. If $h$ has the KL property (of exponent $q$) at every point of ${\rm dom}\,\partial h$, it is called a KL function (of exponent $q$).
\end{definition}

By \cite[Lemma 2.1]{Attouch10}, to prove that a proper lsc function has the KL property (of exponent $q\in[0,1)$), it suffices to check if the property holds at its critical points. As discussed in \cite[Section 4]{Attouch10}, the KL functions are ubiquitous, and the functions definable in an o-minimal structure over the real field admit this property. Moreover, definable functions or mappings conform to the chain rules friendly to optimization.
\section{An inexact MBA method}\label{sec3}

The basic idea of our inexact MBA method is to seek a feasible solution of \eqref{prob} with an improved objective value at each iteration by solving inexactly a strongly convex subproblem, constructed with a local majorization of constraint and objective functions at the current iterate. Before describing its iteration steps, we first take a closer look at the construction of subproblems. 

Let $x^k\in\Gamma$ be the current iterate. Applying Lemma \ref{lemma-major} (i) to every $g_i$ for $i\in[m]$, there exist $\delta_{k}>0$ and $L^k\in\mathbb{R}_{+}^m$ such that for all $x\in\mathbb{B}(x^k,\delta_{k})$ and $V^k\in\overline{\partial}g(x^k)$, 
\[
 g(x)-g(x^k)-(V^k)^{\top}(x-x^k)-\frac{1}{2}\|x-x^k\|^2L^k\in\mathbb{R}_{-}^m.
\]
From \cite[Theorem 10.33]{RW98}, every $g_i$ for $i\in[m]$ is locally expressed as the difference of a $\mathcal{C}^2$-function $f_i$ and a finite convex function, and the constant $L_i^k$ is actually the spectral norm $\|\nabla^2\!f_i(x^k)\|$. For any $(x,s,V,L)\in\mathbb{R}^n\times\mathbb{R}^n\times\mathbb{R}^{n\times m}\times\mathbb{R}_{+}^m$, define 
\begin{equation}\label{def-Gmap}
 G(x,s,V,L):=g(s)+V^{\top}(x\!-\!s)+\frac{1}{2}\|x\!-\!s\|^2L.
\end{equation}
With any $V^k\in\overline{\partial}g(x^k)$, write 
$\Gamma_{\!k}:=\big\{x\in\mathbb{R}^n\,|\,G(x,x^k,V^k,L^k)\in\mathbb{R}_{-}^m\big\}$. Obviously, $\Gamma_k$ is a closed convex set, and
$\Gamma_{\!k}\cap\mathbb{B}(x^k,\delta_{k})\subset\Gamma$ can be regarded as an inner convex approximation to the feasible set $\Gamma$. Motivated by this, we choose $\Gamma_{\!k}$ as the feasible set of the $k$th subproblem. On the other hand, applying Lemma \ref{lemma-major} (i) to the function $g_0$, there exists $L_0^k\ge 0$ such that for any $x$ close enough to $x^k$ and any $\xi^k\in\overline{\partial}g_0(x^k)$, 
\begin{equation*}
 g_0(x)\le g_0(x^k)+\langle\xi^k,x-x^k\rangle+(L_0^k/2)\|x-x^k\|^2.
\end{equation*}
Similarly, $g_0$ is locally expressed as the difference of a $\mathcal{C}^2$-function $f_0$ and a finite convex function, and the constant $L_0^k$ is precisely $\|\nabla^2\!f_0(x^k)\|$. We formulate the objective function of the $k$th subproblem by introducing a positive definite (PD) linear mapping $\mathcal{Q}_k\!:\mathbb{R}^n\to\mathbb{R}^n$ to carry the second-order information of $f_0$ at $x^k$. When $\lambda_{\rm min}(\mathcal{Q}_k)\ge\|\nabla^2\!f_0(x^k)\|$, for any $x$ close enough to $x^k$ and any $\xi^k\in\overline{\partial}g_0(x^k)$, 
\begin{equation*}
 g_0(x)\le \vartheta_k(x):=g_0(x^k)+\langle\xi^k,x-x^k\rangle+\frac{1}{2}\langle x-x^k,\mathcal{Q}_k(x-x^k)\rangle,
\end{equation*}
and the strongly convex function $\vartheta_k$ is a local majorization of $g_0$ at $x^k$. Our method seeks a stationary point of \eqref{prob} by solving a series of strongly convex subproblems 
\begin{equation}\label{subprob}
\min_{x\in\mathbb{R}^n}\big\{F_k(x):=\vartheta_k(x)+\phi(x)\ \ {\rm s.t.}\ \ x\in\Gamma_{\!k}\big\}.
\end{equation}
Take into account that the approximation set $\Gamma_{k}$ is still an intersection of $m$ closed balls. We continue to use the name ``moving balls approximation'' for our method. 

At the $k$th iteration, the approximation effect of \eqref{subprob} to the original problem \eqref{prob} depends on that of $\lambda_{\rm min}(\mathcal{Q}_k)$ to $\|\nabla^2\!f_0(x^k)\|$ and every $L_i^{k}$ to $\|\nabla^2\!f_{i}(x^k)\|$ from above, which has a great influence on the efficiency of algorithms to solve subproblem \eqref{subprob}. Since the spectral norms $\|\nabla^2\!f_0(x^k)\|$ and $\|\nabla^2\!f_{i}(x^k)\|$ are unknown in many scenarios, at each iteration, our MBA method incorporates an inner loop to adaptively search for their tight upper estimations $\lambda_{\rm min}(\mathcal{Q}_k)$ and $L^k$, and meanwhile solves the associated subproblem. Now the subproblems to solve at the $k$th iteration have the form
\begin{align}\label{subprobkj}
 &\min_{x\in\mathbb{R}^n} F_{k,j}(x):=\vartheta_{k,j}(x)+\phi(x)\nonumber\\
 &\ \ {\rm s.t.}\ x\in\Gamma\!_{k,j}:=[G(\cdot,x^k,V^k,L^{k,j})]^{-1}(\mathbb{R}_{-}^m),
\end{align}
where $\vartheta_{k,j}(x):=\langle\xi^k,x\!-\!x^k\rangle+\frac{1}{2}\langle x\!-\!x^k,\mathcal{Q}_{k,j}(x\!-\!x^k)\rangle+g_0(x^k)$ for $x\in\mathbb{R}^n$, and the PD linear operator $\mathcal{Q}_{k,j}:\mathbb{R}^n\to\mathbb{R}^n$ can be chosen by the specific expression of $g_0$.  

Since the exact solution of the subproblem \eqref{subprobkj}, denoted by $\overline{x}^{k,j}$, is unavailable due to computation error or cost. We follow the same line as in \cite{Nabou24} to seek an inexact one. For an inexact solution $y^{k,j}$ of \eqref{subprobkj}, it is natural to require its objective value $F_{k,j}(y^{k,j})$ to be less than $F_{k,j}(x^k)$, which is precisely equal to the current objective value $F(x^k)$ of problem \eqref{prob}, but such a restriction is not enough to achieve satisfactory convergence results due to the lack of multiplier information. To address this issue, similar to \cite{Nabou24}, we require that the KKT residual at $y^{k,j}$ w.r.t. some $(v^{k,j},\lambda^{k,j})\in\partial \phi(y^{k,j})\times\mathbb{R}_{+}^m$ can be controlled by the difference between $y^{k,j}$ and $x^k$. Notice that under a mild CQ, $\overline{x}^{k,j}$ is the unique optimal solution of \eqref{subprobkj} if and only if there exist  $\overline{\lambda}^{k,j}\!\in\mathcal{N}_{\mathbb{R}_{-}^m}(G(\overline{x}^{k,j},x^{k},V^{k},L^{k,j}))$ and $\overline{v}^{k,j}\in\partial\phi(\overline{x}^{k,j})$ such that $\nabla \vartheta_{k,j}(\overline{x}^{k,j})+\overline{v}^{k,j}+\nabla_{\!x} G(\overline{x}^{k,j},x^k,V^k,L^{k,j})\overline{\lambda}^{k,j}=0$. Inspired by this, we define the stationarity residual function $S_{k,j}\!:\mathbb{R}^n\times\mathbb{R}^n\times\mathbb{R}_{+}^m\to\mathbb{R}_{+}$ and the joint complementarity and feasibility violation function  $C_{k,j}:\mathbb{R}^n\times\mathbb{R}_{+}^m\to\mathbb{R}_{+}$ of \eqref{subprobkj} by
\begin{subequations}
 \begin{align}\label{resi-S}
 S_{k,j}(x,v,\lambda)&:=\big\|\nabla\vartheta_{k,j}(x)+v+\nabla\!_{x} G(x,x^k,V^k,L^{k,j})\lambda\big\|,\\
 C_{k,j}(x,\lambda)&:=\big(\!-\!\langle\lambda,G(x,x^k,V^k,L^{k,j})\rangle\big)_{\!+}+\big\|[G(x,x^k,V^k,L^{k,j})]_{+}\big\|_{\infty}.
 \label{resi-C}
 \end{align}
\end{subequations}
From the above discussions, we adopt the following inexactness criterion for $y^{k,j}$:
\begin{subequations}
 \begin{align}
 &F_{k,j}(y^{k,j})\le F_{k,j}(x^k),\ C_{k,j}(y^{k,j},\lambda^{k,j})\le\frac{\beta_C}{2}\|y^{k,j}\!-\!x^k\|^2,\label{inexact1}\\
 &S_{k,j}(y^{k,j},v^{k,j},\lambda^{k,j})\le\!\beta_{S}\|y^{k,j}\!-\!x^k\|.\label{inexact2}
 \end{align}
\end{subequations}

The iterations of our inexact MBA method are described in Algorithm \ref{iMBA}. Observe that $\vartheta_{k,j_k}=\vartheta_k, F_{k,j_k}=F_k$ and $\Gamma_{k,j_k}=\Gamma_k$. Therefore, to simplify the notation, we denote $S_{k,j_k}, C_{k,j_k}$ and $\overline{x}^{k,j_k}$ as $S_k, C_k$ and $\overline{x}^k$, respectively.

\begin{algorithm}
\caption{\label{iMBA}{\bf (Inexact MBA method for problem \eqref{prob})}}
\begin{algorithmic}
\STATE{\textbf{Input}: $0<\!\mu_{\min}\le\mu_{\rm max},0<\!L_{\min}\le\! L_{\max},M>0,\beta_C>0,\beta_S>0,\alpha>0,$ \hspace*{1.1cm} $\tau\!>1$ and an initial $x^0\in\Gamma$.}
\FOR{$k=0,1,2,\ldots$}
\STATE{1. Choose $\xi^k\in\partial g_0(x^k), V^k\in\partial g(x^k),\mu_{k,0}\in[\mu_{\min},\mu_{\max}],L^{k,0}\in[L_{\rm min},L_{\rm max}]$.}
\STATE{2. \textbf{for} $j=0,1,2,\ldots$ \textbf{do}
\begin{itemize}
 \item[(2a)] Choose a linear operator $\mathcal{Q}_{k,j}$ with $\mu_{k,j}\mathcal{I}\preceq \mathcal{Q}_{k,j}\preceq (\mu_{k,j}\!+\!M)\mathcal{I}$ to formulate the subproblem \eqref{subprobkj}, and then solve it to seek a vector $y^{k,j}$ with some $(v^{k,j},\lambda^{k,j})\in\partial\phi(y^{k,j})\times\mathbb{R}^m_+$  such that \eqref{inexact1}-\eqref{inexact2}.	
 
 \item [(2b)] If $g(y^{k,j})\in\mathbb{R}_{-}^m$ and $F(y^{k,j})\le F(x^k)-\frac{\alpha}{2}\|y^{k,j}-x^k\|^2$, set $j_k\!:=j$ and go to step 3; else if $g(y^{k,j})\notin\mathbb{R}_{-}^m$, go to step (2c); else go to step (2d).
			
 \item [(2c)] Set $L^{k,j+1}=\tau L^{k,j}$ and $\mu_{k,j+1}=\mu_{k,j}$.
			
 \item [(2d)] Set $L^{k,j+1}=L^{k,j}$ and $\mu_{k,j+1}=\tau\mu_{k,j}$.
 \end{itemize}
 \textbf{end for}}
 
\STATE{3. Let $(\mu_k,\mathcal{Q}_k,L^k,x^{k+1},v^{k+1},\lambda^{k+1})\!:=\!(\mu_{k,j_k},\mathcal{Q}_{k,j_k},L^{k,j_k},y^{k,j_k},v^{k,j_k},\lambda^{k,j_k})$.}\label{step3}
\ENDFOR
\end{algorithmic}
\end{algorithm}
\begin{remark}\label{remark-alg}
 {\bf (a)} Algorithm \ref{iMBA} is well defined by Lemma \ref{lemma-welldef} below. Its inner loop seeks tight upper estimations $\mu_k$ for $\|\nabla^2\!f_0(x^k)\|$ and $L_i^k$ for $\|\nabla^2\!f_i(x^k)\|$, and meanwhile, computes an inexact solution $y^{k,j}$ satisfying \eqref{inexact1}-\eqref{inexact2}. Our inexactness criterion \eqref{inexact1}-\eqref{inexact2} is a little weaker than the one proposed in \cite{Nabou24} for problem \eqref{prob} with $L$-smooth $g_i$ for $i\in[m]_{+}$. Indeed, the first inequality in \eqref{inexact1} coincides with \cite[Eq.(3.4)]{Nabou24}, but our complementarity violation in $C_{k,j}$ only involves that of the feasible constraints, so the second inequality in \eqref{inexact1} along with \eqref{inexact2} is weaker than the one in \cite[Assumption 4]{Nabou24}.

 \noindent
 {\bf(b)} The computation cost of Algorithm \ref{iMBA} at each iteration is not expensive despite involving an inner loop. Observe that each step of the inner loop does not involve seeking any element of $\partial g_i(x^{k,j})$ for $i\in[m]_{+}$, and when first-order algorithms are applied to solve the subproblems, their iterations just use the data from the outer loop, so the computation cost in the inner loop is cheap. Also, Lemma \ref{lemma-welldef} (ii) below shows that the number of steps in the inner loop can be quantified in terms of the logarithms of the upper-$\mathcal{C}^2$ constants of $g_i$ for $i\in[m]_{+}$ on a compact set.
	
 \noindent
 {\bf(c)} We claim that $x^{k}$ is a stationary point of \eqref{prob} in the sense of Definition \ref{spoint-def} once $x^{k}=x^{k+1}$. Indeed, if $x^{k}=x^{k+1}$, by the definition of $x^{k+1}$ in step 3, $y^{k,j_k}=x^k$. Together with the inexactness conditions \eqref{inexact1}-\eqref{inexact2}, $C_{k,j_k}(x^{k+1},\lambda^{k+1})=0$ and $S_{k,j_k}(x^{k+1},v^{k+1},\lambda^{k+1})=0$. From the expressions of $C_{k,j}$ and $S_{k,j}$ in \eqref{resi-S}-\eqref{resi-C},  
 \begin{subequations}
 \begin{align*}
 \nabla\vartheta_{k,j_k}(x^{k+1})+v^{k+1}+\nabla\!_{x} G(x^{k+1},x^{k},V^{k},L^{k})\lambda^{k+1}=0,\qquad\qquad\\
 \lambda^{k+1}\in\mathbb{R}_{+}^m,\,G(x^{k+1},x^{k},V^{k},L^{k})\in\mathbb{R}_{-}^m,\,\langle\lambda^{k+1},G(x^{k+1},x^{k},V^{k},L^{k})\rangle=0.
 \end{align*}
 \end{subequations}
 Then, using $x^{k+1}=x^k$ and the expressions of $\vartheta_{k,j_{k}}$ and $G(\cdot,x^{k},V^{k},L^{k})$ leads to
 \[
   0\in\xi^{k}+\partial\phi(x^{k})+V^{k}\lambda^{k+1}\ \ {\rm and}\ \ \lambda^{k+1}\in\mathcal{N}_{\mathbb{R}_{-}^m}(g(x^{k})).
 \]
 This, along with $\xi^{k}\in\partial g_0(x^{k})$ and $V^k\in\partial g(x^k)$, shows that $x^{k}$ (of course $x^{k+1}$) is a stationary point of problem \eqref{prob} in the sense of Definition \ref{spoint-def}. In view of this, the condition $\|x^{k+1}-x^{k}\|\le\epsilon_1$ for a tolerance $\epsilon_1$ is suggested as a stop-condition. Since Algorithm \ref{iMBA} produces at each iteration a feasible point $x^k\in\Gamma$ and a nonnegative $\lambda^k$, one can terminate its iteration once $[-\langle\lambda^k,g(x^k)\rangle]_{+}\le\epsilon_2$ for a tolerance $\epsilon_2$. 
\end{remark}

To ensure that Algorithm \ref{iMBA} is well defined, we need the following assumption. 
\begin{assumption}\label{ass2}
 For each $k,j\in\mathbb{N}$, the mapping $\mathcal{G}_{k,j}(\cdot):=G(\cdot,x^k,V^k,L^{k,j})-\mathbb{R}_{-}^m$ is subregular at $(\overline{x}^{k,j},0)$, where $\overline{x}^{k,j}$ is the unique optimal solution of \eqref{subprobkj}. 
\end{assumption}

Although Assumption \ref{ass2} is made for the subproblems, following the proof of \cite[Proposition 2.1 (iii)]{Auslender10}, one can verify that the MFCQ at $x^k$ for \eqref{prob} implies the Slater's CQ for \eqref{subprobkj}, which further implies Assumption \ref{ass2}. From the discussion after \eqref{mapFH}, it is actually the weakest CQ for the multiplier set of \eqref{subprobkj} to be nonempty since, for each $k,j\in\mathbb{N}$, the components of $G(\cdot,x^k,V^k,L^{k,j})$ are smooth and convex. 
\begin{lemma}\label{lemma-welldef}
Under Assumption \ref{ass2}, if $x^k\in\Gamma$ is a non-stationary point, then 
\begin{itemize}
\item[(i)] the inexactness conditions \eqref{inexact1}-\eqref{inexact2} are well-defined; 

\item[(ii)] the inner loop stops once $j>\lceil\max\limits_{i\in[m]}\frac{\log[(\beta_{C}+\rho_i^k)/L_i^{k,0}]}{\log\tau}\rceil+\frac{\log[(\alpha+\rho_0^k)/\mu_{k,0}]}{\log\tau}$ where, for each $i\in[m]_+$, $\rho_i^k>0$ is a constant depending only on $g_i$ and the set 
\[
X_k\!:=\big\{x\in\mathbb{R}^n\,|\,({\mu_{\rm min}}/{2})\|x-x^k\|^2+\langle\xi^k,x-x^k\rangle+\phi(x)\le\phi(x^k)\big\}. 
\]
\end{itemize}
Hence, Algorithm \ref{iMBA} is well-defined with $\{x^k\}_{k\in\mathbb{N}}\subset\Gamma$ and $x^k\in\Gamma_{k,j}$ for all $j\in\mathbb{N}$.
\end{lemma}
\begin{proof}
From $x^k\in\Gamma$ and $G(x^k,x^k,V^k,L^{k,j})=g(x^k)$, the set $\Gamma_{k,j}$ is nonempty. Along with the strong convexity of $F_{k,j}$,  the optimal solution $\overline{x}^{k,j}$ of \eqref{subprobkj} exists. 
	
\noindent
{\bf(i)} Fix any $j\in\mathbb{N}$. To prove that the conditions \eqref{inexact1}-\eqref{inexact2} are well-defined, it suffices to consider $\overline{x}^{k,j}\neq x^k$. If not, $x^{k}$ is an optimal solution of \eqref{subprobkj}, so we get
\[
0\in\partial F_{k,j}(x^k)+\mathcal{N}_{\Gamma_{k,j}}(x^k)=\xi^k+\partial\phi(x^k)+\mathcal{N}_{\Gamma_{k,j}}(x^k).
\]
In view of Assumption \ref{ass2}, $\mathcal{N}_{\Gamma_{k,j}}(x^k)=\nabla_{\!x} G(x^{k},x^k,V^k,L^{k,j})\mathcal{N}_{\mathbb{R}_{-}^m}(G(x^{k},x^k,V^k,L^{k,j}))$. Note that $\nabla\!_{x}G(x^{k},x^k,V^k,L^{k,j})=V^k$ and  $G(x^{k},x^k,V^k,L^{k,j})=g(x^k)$. Then,  
\[
 0\in\xi^k+\partial\phi(x^k)+V^k\mathcal{N}_{\mathbb{R}_{-}^m}(g(x^{k})). 
\]
This by Definition \ref{spoint-def} shows that $x^k$ is a stationary point of \eqref{prob}, which is impossible. 
Now from the optimality condition of \eqref{subprobkj} at $\overline{x}^{k,j}$, $0\in\partial F_{k,j}(\overline{x}^{k,j})+\mathcal{N}_{\Gamma_{k,j}}(\overline{x}^{k,j})$ which, by the expression of $F_{k,j}$ and the above arguments, can equivalently be written as
\[
 0\in\nabla\vartheta_{k,j}(\overline{x}^{k,j})+\partial\phi(\overline{x}^{k,j})+\nabla_{\!x}G(\overline{x}^{k,j},x^k,V^k,L^{k,j})\mathcal{N}_{\mathbb{R}_{-}^m}(G(\overline{x}^{k,j},x^k,V^k,L^{k,j})). 
\]
Then, there exist  $\overline{v}^{k,j}\in\partial\phi(\overline{x}^{k,j})$ and  $\overline{\lambda}^{k,j}\in\mathbb{R}^m$ such that $(\overline{x}^{k,j},\overline{v}^{k,j},\overline{\lambda}^{k,j})$ satisfies  
\begin{subequations}
\begin{align}\label{KKT1-subprob}
 \nabla\vartheta_{k,j}(\overline{x}^{k,j})+\overline{v}^{k,j}+\nabla_{\!x}G(\overline{x}^{k,j},x^k,V^k,L^{k,j})\overline{\lambda}^{k,j}=0,\qquad\qquad\\
 \label{KKT2-subprob}
 \overline{\lambda}^{k,j}\in\mathbb{R}_{+}^m,\,G(\overline{x}^{k,j},x^k,V^k,L^{k,j})\in\mathbb{R}_{-}^m,\,\langle \overline{\lambda}^{k,j},G(\overline{x}^{k,j},x^k,V^k,L^{k,j})\rangle=0. 
\end{align}
\end{subequations} 
From \eqref{KKT1-subprob}-\eqref{KKT2-subprob} and $\overline{x}^{k,j}\ne x^k$, we have $S_{k,j}(\overline{x}^{k,j},\overline{v}^{k,j},\overline{\lambda}^{k,j})=0<\beta_{S}\|\overline{x}^{k,j}-x^k\|$ and $C_{k,j}(\overline{x}^{k,j},\overline{\lambda}^{k,j})=0<\frac{\beta_{C}}{2}\|\overline{x}^{k,j}-x^k\|^2$. Recall that $\overline{v}^{k,j}\in\partial\phi(\overline{x}^{k,j})$ and $\partial\phi$ is osc at $\overline{x}^{k,j}$. There is a sequence $\{(x^{l},v^{l})\}_{l\in\mathbb{N}}\subset\mathbb{R}^n\times\partial\phi(x^{l})$ with $(x^{l},v^{l})\to(\overline{x}^{k,j},\overline{v}^{k,j})$. Note that $F_{k,j}(\overline{x}^{k,j})<F_{k,j}(x^k)$. By the continuity of $F_{k,j}, S_{k,j}(\cdot,\cdot,\overline{\lambda}^{k,j})$ and $C_{k,j}(\cdot,\overline{\lambda}^{k,j})$, there exist $x$ close to $\overline{x}^{k,j}$ and $v\in\partial\phi(x)$ such that $(x,v,\overline{\lambda}^{k,j})$ satisfies \eqref{inexact1}-\eqref{inexact2}. 

\noindent
{\bf(ii)} Obviously, $x^k\in X_k$. From the first inequality of \eqref{inexact1}, the expression of $F_{k,j}$ and $\mathcal{Q}_{k,j}\succeq\mu_{\min}\mathcal{I}$, it is easy to check that $\{y^{k,j}\}_{j\in\mathbb{N}}\subset X_k$. 
By Assumption \ref{ass1} (ii), there exists $\widehat{x}\in\Gamma$ and $\widehat{\zeta}\in\partial\phi(\widehat{x})$ such that $\phi(x)\ge\phi(\widehat{x})+\langle\widehat{\zeta},x-\widehat{x}\rangle$ for all $x\in\mathbb{R}^n$. Then, the function $\mathbb{R}^n\ni x\mapsto\frac{\mu_{\rm min}}{2}\|x-x^k\|^2+\langle\xi^k,x-x^k\rangle+\phi(x)$ is coercive, which implies that the set $X_k$ is compact.
By Assumption \ref{ass1} and Lemma \ref{lemma-major} (ii) with $\mathcal{O}'=\mathbb{R}^n,D=X_k$, for each $i\in[m]$, there exists $\rho_i^k>0$ such that for all $x,y\in X_k$ and $\zeta^{i}\in\overline{\partial}g_i(x)$,
\begin{equation}\label{quant-0}
 g_i(y)\le g_i(x)+\langle\zeta^{i},y-x\rangle+({\rho_i^k}/{2})\|y-x\|^2.
\end{equation}
For each $j\in\mathbb{N}$, combining $\|[G(y^{k,j},x^k,V^k,L^{k,j})]_+\|_{\infty}\le C_{k,j}$ with the second inequality of \eqref{inexact1} yields that $g_i(x^k)+ \langle V_{i}^k,y^{k,j}-x^k\rangle+\frac{L^{k,j}_i-\beta_{C}}{2}\|y^{k,j}-x^k\|^2\le0$ for all $i\in[m]$. Now for each $i\in[m]$, 
from \eqref{quant-0} the following inequality holds for any $j\in\mathbb{N}$
\begin{align*}
 g_i(y^{k,j})\le g_i(x^k)\!+\!\langle V_{i}^k,y^{k,j}\!-\!x^k\rangle\!+\!\frac{\rho_i^k}{2}\|y^{k,j}\!-\!x^k\|^2
 \le\frac{1}{2}[\rho_i^k\!-\!(L^{k,j}_i\!-\!\beta_{C})]\|y^{k,j}\!-\!x^k\|^2,
\end{align*}
so $y^{k,j}\in\Gamma$ once $\beta_{C}+\rho_i^k\le L^{k,j}_i=\tau^{j}L_i^{k,0}$. Then, the first inequality in step (2b) holds as long as $j>j_1:=\lceil\max\limits_{i\in[m]}\frac{\log(\beta_{C}+\rho_i^k)-\log L_i^{k,0}}{\log\tau}\rceil$. For each $j\ge j_1$, since $y^{k,j},x^k\in\Gamma\cap X_k$, applying Lemma \ref{lemma-major} (ii) with $\mathcal{O}'=\mathcal{O},D=X_k\cap\Gamma$, there exists $\rho_0^k>0$ such that 
\begin{equation*}
 g_0(y^{k,j})\le g_0(x^k)+\langle\xi^k,y^{k,j}-x^k\rangle+({\rho_0^k}/{2})\|y^{k,j}-x^k\|^2.
\end{equation*}  
Together with the expressions of $F$ and $F_{k,j}$, it follows that for each $j\ge j_1$, 
\begin{align}\label{inner-f}  F(y^{k,j})&=g_0(y^{k,j})+\phi(y^{k,j})\le g_0(x^k)+\langle\xi^k,y^{k,j}-x^k\rangle+\frac{\rho_0^k}{2}\|y^{k,j}-x^k\|^2+\phi(y^{k,j})\nonumber\\
  &=F_{k,j}(y^{k,j})+\frac{\rho_0^k}{2}\|y^{k,j}\!-\!x^k\|^2-\frac{1}{2}\langle y^{k,j}\!-\!x^k,\mathcal{Q}_{k,j}(y^{k,j}\!-\!x^k)\rangle\nonumber\\
  &\stackrel{\eqref{inexact1}}{\le} F(x^k)-\frac{\alpha}{2}\|y^{k,j}-x^k\|^2+\frac{\rho_0^k+\alpha-\tau^{j-j_1}\mu_{k,0}}{2}\|y^{k,j}-x^k\|^2,
 \end{align}
 where the second inequality is due to $\mathcal{Q}_{k,j}\succeq\mu_{k,j}\mathcal{I}=\tau^{j-j_1}\mu_{k,j_1}\mathcal{I}\succeq\tau^{j-j_1}\mu_{k,0}\mathcal{I}$. Thus, the two inequalities in step (2b) hold when $\rho_0^k\!+\!\alpha\le \tau^{j-j_1}\mu_{k,0}$ or equivalently $j>\frac{\log(\rho_0^k+\alpha)-\log\mu_{k,0}}{\log\tau}+j_1$. That is, the inner loop stops once $j>\frac{\log(\rho_0^k+\alpha)-\log\mu_{k,0}}{\log\tau}+j_1$. 

Notice that Algorithm \ref{iMBA} starts with $x^0\in\Gamma$, and the inner loop returns the point $y^{k,j_k}=x^{k+1}\in\Gamma$. Then, the arguments by induction can prove that for each $k\in\mathbb{N}$,  the subproblem \eqref{subprobkj} and the conditions \eqref{inexact1}-\eqref{inexact2} are well-defined. Thus, Algorithm \ref{iMBA} is well-defined with $\{x^k\}_{k\in\mathbb{N}}\subset\Gamma$ and $x^k\in\Gamma_{k,j}$ for all $j\in\mathbb{N}$. 
\end{proof}

From the iterations of the outer loop in Algorithm \ref{iMBA}, it is immediate to have 
\begin{equation}\label{F-decrease}
 F(x^{k+1})\le F(x^k)-({\alpha}/{2})\|x^{k+1}-x^k\|^2\quad\ \forall k\in\mathbb{N}.
\end{equation}
Then, the lower boundedness of $F$ on $\Gamma$ by Assumption \ref{ass1} (ii) implies the convergence of $\{F(x^k)\}_{k\in\mathbb{N}}$, so $\sum_{k=0}^{\infty}\|x^{k+1}\!-x^k\|^2<\infty$ follows. 
That is, the following result holds.
\begin{corollary}\label{corollary-objF}
Under Assumption \ref{ass2}, the above \eqref{F-decrease} holds for all $k\in\mathbb{N}$, so the sequence $\{F(x^k)\}_{k\in\mathbb{N}}$ converges to some $\varpi^*\in\mathbb{R}$ and $\sum_{k=0}^{\infty}\|x^{k+1}\!-x^k\|^2<\infty$.
\end{corollary}
\section{Convergence analysis}\label{sec4}

This section is dedicated to the convergence analysis of Algorithm \ref{iMBA} under Assumption \ref{ass3}, a common one for analyzing convergence of iterate sequences. Since $\{x^k\}_{k\in\mathbb{N}}\subset\mathcal{L}_{F(x^0)}\!:=\{x\in\mathbb{R}^n\,|\,F(x)\le F(x^0)\}$ by \eqref{F-decrease}, this assumption automatically holds if the level set $\mathcal{L}_{F(x^0)}$ is bounded. 
\begin{assumption}\label{ass3}
$\{x^k\}_{k\in\mathbb{N}}$ is bounded, and denote its  cluster point set as $\omega(x^0)$.
\end{assumption}

The following two technical lemmas are frequently used in the subsequent analysis. Lemma \ref{lemma-bound} verifies the boundedness of the sequence $\{(\xi^k,\overline{x}^k,V^k,L^k,\mathcal{Q}_k)\}_{k\in\mathbb{N}}$, and Lemma \ref{lemma-xbark} states that for large enough $k$ the iterates $x^k$ and $x^{k+1}$ are close to $\overline{x}^k$. 
\begin{lemma}\label{lemma-bound}
 Under Assumptions \ref{ass2}-\ref{ass3}, the following  statements hold true. 
 \begin{itemize}
 \item[(i)] The sequence $\{(\overline{x}^k,\xi^k,V^k)\}_{k\in\mathbb{N}}$ is bounded.
		
 \item[(ii)] There exists $\beta_{L}>0$ such that $\|L^k\|\le\beta_{L}$ for all $k$.
		
 \item[(iii)] $\{\mu_k\}_{k\in\mathbb{N}}$ is bounded, so there is $\beta_{\mathcal{Q}}>0$ such that $\|\mathcal{Q}_k\|\!\le\beta_{\mathcal{Q}}$ for all $k\in\mathbb{N}$. 
\end{itemize}
\end{lemma}
\begin{proof} 
 {\bf (i)} From $\xi^k\in\partial g_0(x^k)$ and $V^k\in\partial g(x^k)$, the boundedness of $\{(\xi^k,V^k)\}_{k\in\mathbb{N}}$ follows $\{x^k\}_{k\in\mathbb{N}}\subset\Gamma$ and Remark \ref{remark-subdiff} (b). Next we prove that $\{\overline{x}^k\}_{k\in\mathbb{N}}$ is bounded. For each $k\in\mathbb{N}$, using the expression of $F_k$ and $\mathcal{Q}_{k}\succeq\mu_{\rm min}\mathcal{I}$ leads to 
 \[
  g_0(x^k)+\langle\xi^k,\overline{x}^k\!-\!x^k\rangle +\frac{\mu_{\rm min}}{2}\|\overline{x}^k\!-\!x^k\|^2+\phi(\overline{x}^k)\le F_k(\overline{x}^k)\le F(x^k)\stackrel{\eqref{F-decrease}}{\le} F(x^0),
 \]
 where the second inequality is due to $F_k(\overline{x}^k)\le F_k(x^k)=F(x^k)$. By Assumption \ref{ass1} (ii), there exists $v^0\in\partial\phi(x^0)$ such that $\phi(\overline{x}^k)\ge\phi(x^0)+\langle v^0,\overline{x}^k-x^0\rangle$. Then, 
 \[
  g_0(x^k)+\langle\xi^k,\overline{x}^k\!-\!x^k\rangle +({\mu_{\rm min}}/{2})\|\overline{x}^k\!-\!x^k\|^2+\phi(x^0)+\langle v^0,\overline{x}^k-x^0\rangle\le  F(x^0).
 \] 
 This, by the boundedness of $\{(x^k,\xi^k)\}_{k\in\mathbb{N}}$ and $\{g_0(x^k)\}_{k\in\mathbb{N}}$, implies that of $\{\overline{x}^k\}_{k\in\mathbb{N}}$. 
	
\noindent
{\bf (ii)} 
 Suppose on the contrary that $\{L^k\}_{k\in\mathbb{N}}\subset\mathbb{R}_{++}^m$ is unbounded. For each $k\in\mathbb{N}$, let $\widehat{L}^k:=L^k/\tau$. Clearly, $\lim_{k\to\infty}\|\widehat{L}^k\|_{\infty}=\infty$. From steps (2b)-(2c), there exist an infinite index $\mathcal{K}$, an index $i_1\in[m]$ such that  $\|\widehat{L}^k\|_{\infty}=\widehat{L}_{i_1}^k$ for all $k\in\mathcal{K}$, and an index $1<\widehat{j}_k<j_k$ such that $\widehat{L}^k=L^{k,\widehat{j}_k}$ and $g_{i_1}(y^{k,\widehat{j}_k})>0$ for all $k\in\mathcal{K}$. Since $\{x^k\}_{k\in\mathcal{K}}\subset\Gamma$ is bounded, if necessary by taking a subsequence, we can assume  $\lim_{\mathcal{K}\ni k\to\infty}x^k=x^*\in\Gamma$. Applying Lemma \ref{lemma-major} (i) to the function $g_{i_1}$ with $x=x^*$, there exist $\delta>0$ and $\rho_{i_1}^{*}>0$  such that for any $x',x''\in\mathbb{B}(x^*,\delta)$ and $\zeta'\in\overline{\partial} g_{i_1}(x')$,
\begin{equation}\label{upperC2}
 g_{i_1}(x'')\le g_{i_1}(x')+\langle\zeta',x''-x'\rangle+({\rho_{i_1}^*}/{2})\|x''-x'\|^2.
\end{equation}
For each $k\in\mathcal{K}$, since $\widehat{L}^k=L^{k,\widehat{j}_k}$ and  $y^{k,\widehat{j}_k}$ is an inexact minimizer of \eqref{subprobkj} with $j=\widehat{j}_k$ satisfying \eqref{inexact1}, from the expression of the function $C_{k,\widehat{j}_k}$ it follows
\begin{align}\label{Lk-temp2}    
0&\ge\|[G(y^{k,\widehat{j}_k},x^k,V^k,\widehat{L}^k)]_{+}\|_{\infty}-
    ({\beta_{C}}/{2})\|y^{k,\widehat{j}_k}-x^k\|^2\nonumber\\
&\ge [G_{i_1}(y^{k,\widehat{j}_k},x^k,V^k,\widehat{L}^k)]_{+}-
    ({\beta_{C}}/{2})\|y^{k,\widehat{j}_k}-x^k\|^2\nonumber\\
&\ge g_{i_1}(x^k)+ \langle V_{i_1}^k,y^{k,\widehat{j}_k}-x^k\rangle+[({\widehat{L}^k_{i_1}-\beta_{C}})/{2}]\|y^{k,\widehat{j}_k}-x^k\|^2.
\end{align}
Since $\{V_{i_1}^k\}_{k\in\mathcal{K}}$ is bounded by part (i) and $\lim_{\mathcal{K}\ni k\to\infty}\widehat{L}_{i_1}^k=\infty$, we infer from \eqref{Lk-temp2} that $\lim_{\mathcal{K}\ni k\to\infty}\|y^{k,\widehat{j_k}}-x^k\|=0$. Recall that $\lim_{\mathcal{K}\ni k\to\infty}x^k=x^*$. There exists $\overline{k}\in\mathbb{N}$ such that for all $\mathcal{K}\ni k\ge\overline{k}$, $y^{k,\widehat{j_k}},x^k\in\mathbb{B}(x^*,\delta)$. Now, for each $\mathcal{K}\ni k\ge\overline{k}$, invoking the above \eqref{upperC2} with $(x'',x')=(y^{k,\widehat{j}_k}, x^k)$ and $\zeta'=V_{i_1}^k$ immediately results in 
\[
 0<g_{i_1}(y^{k,\widehat{j}_k})\le g_{i_1}(x^k)+ \langle V_{i_1}^k,y^{k,\widehat{j}_k}-x^k\rangle+({\rho_{i_1}^*}/{2})\|y^{k,\widehat{j}_k}-x^k\|^2.
\]
Together with \eqref{Lk-temp2}, for each $\mathcal{K}\ni k\ge\overline{k}$, it holds
$({\widehat{L}^k_{i_1}-\beta_{C}}-\rho_{i_1}^*)\|y^{k,\widehat{j}_k}-x^k\|^2< 0$. Note that $y^{k,\widehat{j_k}}\neq x^k$ for all $k\in\mathcal{K}$ as $g_{i_1}(y^{k,\widehat{j_k}})>0\ge g_{i_1}(x^k)$. We get $\widehat{L}_{i_1}^k<\beta_{C}+\rho_{i_1}^*$ for all $\mathcal{K}\ni k\ge\overline{k}$, a contradiction to  $\lim_{\mathcal{K}\ni k\to\infty}\widehat{L}^k_{i_1}=\infty$. 
	
\noindent
{\bf(iii)} Suppose on the contrary that $\{\mu_k\}_{k\in\mathbb{N}}$ is unbounded. Let $\widehat{\mu}_k:=\tau^{-1}\mu_k$ for each $k\in\mathbb{N}$. From step (2d) and the boundedness of $\{x^k\}_{k\in\mathbb{N}}$, there exists  $\mathcal{K}\subset\mathbb{N}$ such that $\lim_{\mathcal{K}\ni k\to\infty}x^k=x^*\in\Gamma$ and for each $k\in\mathcal{K}$, an index $1<\widehat{j}_k<j_k$ exists with 
\begin{equation}\label{ineq-mubound0}
\widehat{\mu}_k=\mu_{k,\widehat{j}_k}\ \ {\rm and}\ \ F(y^{k,\widehat{j}_k})>F(x^k)-({\alpha}/{2})\|y^{k,\widehat{j}_k}-x^k\|^2.
\end{equation}
Applying Lemma \ref{lemma-major} (i) with $\mathcal{O}'=\mathcal{O}$ to the function $g_0$ for $x=x^*$, there exist $\delta>0$ and $\rho_{0}^*>0$ such that for any $x',x''\in\mathbb{B}(x^*,\delta)$ and $\xi\in\overline{\partial} g_0(x')$,
\begin{equation}\label{upperC2-F0}
g_0(x'')\le g_0(x') +\langle\xi,x''-x'\rangle+({\rho_{0}^*}/{2})\|x''-x'\|^2.
\end{equation}
Note that for each $k\in\mathcal{K}$, $\widehat{\mu}_k\mathcal{I}=\mu_{k,\widehat{j}_k}\mathcal{I}\preceq\mathcal{Q}_{k,\widehat{j}_k}$ and  $y^{k,\widehat{j}_k}$ is an inexact minimizer of \eqref{subprobkj} with $j=\widehat{j_k}$ satisfying \eqref{inexact1}. From the first inequality of \eqref{inexact1}, for any $k\in\mathcal{K}$, 
\begin{align*}
(\widehat{\mu}_k/2)\|y^{k,\widehat{j}_k}-x^k\|^2&\le-\langle\xi^k,y^{k,\widehat{j}_k}-x^k\rangle+\phi(x^k)-\phi(y^{k,\widehat{j}_k})\\
&\le-\langle\xi^k+v^k,y^{k,\widehat{j}_k}-x^k\rangle\ {\rm with}\ v^k\in\partial\phi(x^k).
\end{align*}
From part (i) and Remark \ref{remark-subdiff} (b) , the sequence $\{(\xi^k,v^k)\}_{k\in\mathbb{N}}$ is bounded. Then, the above inequality implies $\lim_{\mathcal{K}\ni k\to\infty}\|y^{k,\widehat{j}_k}-x^k\|=0$. Since $\lim_{\mathcal{K}\ni k\to\infty}x^k=x^*$, there exists $\overline{k}\in\mathbb{N}$ such that $y^{k,\widehat{j}_k},x^k\in\mathbb{B}(x^*,\delta)$ for any $\mathcal{K}\ni k\ge\overline{k}$. 
Now applying \eqref{upperC2-F0} with $x''=y^{k,\widehat{j}_k}$, $x'=x^k$ and $\xi=\xi^k$, we infer that for any $\mathcal{K}\ni k\ge\overline{k}$,
\begin{align}\label{muk-temp1}
F(y^{k,\widehat{j}_k})&\le g_0(x^k)+\langle\xi^k,y^{k,\widehat{j}_k}-x^k\rangle+({\rho_0^*}/{2})\|y^{k,\widehat{j}_k}-x^k\|^2+\phi(y^{k,\widehat{j}_k})\nonumber\\
&\le F_{k,\widehat{j_k}}(y^{k,{\widehat{j}_k}})+\frac{\rho_0^*\!-\!\widehat{\mu}_k}{2}\|y^{k,\widehat{j}_k}\!-\!x^k\|^2
\stackrel{\eqref{inexact1}}{\le}F(x^k)+\frac{\rho_0^*\!-\!\widehat{\mu}_k}{2}\|y^{k,\widehat{j}_k}\!-\!x^k\|^2.
\end{align}
Note that for all $\mathcal{K}\ni k\ge\overline{k}$, $y^{k,\widehat{j}_k}\neq x^k$ from \eqref{ineq-mubound0}. Together with \eqref{ineq-mubound0}, we obtain $\widehat{\mu}_k<\alpha+\rho_0^*$, which is a contradiction to $\lim_{k\to\infty}\widehat{\mu}_k=\infty$.
\end{proof} 
\begin{lemma}\label{lemma-xbark}
 Under Assumptions \ref{ass2}-\ref{ass3}, $\omega(x^0)$ is nonempty and compact, and  
 \begin{equation*}
 \lim_{k\to\infty}\|x^k-\overline{x}^k\|=0\ \ {\rm and}\ \ \lim_{k\to\infty}\|x^{k+1}-\overline{x}^k\|=0.
\end{equation*}
\end{lemma}  
\begin{proof}
It suffices to prove the first limit, since the second follows the first one and Corollary \ref{corollary-objF}. For each $k\in\mathbb{N}$, from the optimality of $\overline{x}^k$ to \eqref{subprob}, there exists $\overline{v}^k\in\partial\phi(\overline{x}^k)$ such that $-\nabla\vartheta_k(\overline{x}^k)-\overline{v}^k\in\mathcal{N}_{\Gamma_k}(\overline{x}^k)$. Along with the convexity of $\Gamma_k$, 
\begin{equation*}
 0\ge \langle\nabla\vartheta_k(\overline{x}^k)+\overline{v}^k,\overline{x}^k-x^k\rangle=\langle\xi^k+\overline{v}^k+\mathcal{Q}_k(\overline{x}^k\!-\!x^k),\overline{x}^k\!-\!x^k\rangle.
\end{equation*}
Combining this inequality with the expression of $F_k$ and $\overline{v}^k\in\partial\phi(\overline{x}^k)$ leads to 
\begin{align*}
 F_k(x^k)-F_k(\overline{x}^k)&=\phi(x^k)-\langle\xi^k,\overline{x}^k\!-\!x^k\rangle-\frac{1}{2}\langle\overline{x}^k\!-\!x^k,\mathcal{Q}_k(\overline{x}^k\!-\!x^k)\rangle-\phi(\overline{x}^k)\\
 &\ge\frac{1}{2}\langle\overline{x}^k-x^k,\mathcal{Q}_k(\overline{x}^k-x^k)\rangle+\phi(x^k)-\phi(\overline{x}^k)+\langle\overline{v}^k,\overline{x}^k-x^k\rangle\\
&\ge \frac{1}{2}\langle\overline{x}^k-x^k,\mathcal{Q}_k(\overline{x}^k-x^k)\rangle\ge\frac{\mu_{\min}}{2}\|\overline{x}^k-x^k\|^2\quad\forall k\in\mathbb{N}.
\end{align*}
Let $L_k(x,\lambda):=F_k(x)+\langle\lambda,G(x,x^k,V^k,L^k)\rangle$ for $(x,\lambda)\in\mathbb{R}^n\times\mathbb{R}^m$ be the Lagrange function of \eqref{subprob}. From the convexity of $L_k(\cdot,\lambda^{k+1})$ and the inclusion 
$\nabla\vartheta_k(x^{k+1})+v^{k+1}+\nabla\!_xG(x^{k+1},x^k,V^k,L^k)\lambda^{k+1}\in\partial_x L_k(x^{k+1},\lambda^{k+1})$, it follows 
$L_k(x^{k+1},\lambda^{k+1})-L_k(\overline{x}^k,\lambda^{k+1})\le\langle \nabla\vartheta_k(x^{k+1})+v^{k+1}+\nabla\!_xG(x^{k+1},x^k,V^k,L^k)\lambda^{k+1},x^{k+1}-\!\overline{x}^k\rangle$. Then, 
\begin{align}\label{sub-optval}
 F_k(x^{k+1})\!-\!F_k(\overline{x}^k)   &\le\langle\nabla\vartheta_k(x^{k+1})\!+\!v^{k+1}\!+\!\nabla\!_xG(x^{k+1},x^k,V^k,L^k)\lambda^{k+1},x^{k+1}\!-\!\overline{x}^k\rangle\nonumber\\
 &\quad-\langle\lambda^{k+1},G(x^{k+1},x^k,V^k,L^k)\rangle+\langle\lambda^{k+1},G(\overline{x}^k,x^k,V^k,L^k)\rangle\nonumber\\
 &\le S_k(x^{k+1},v^{k+1},\lambda^{k+1})\|x^{k+1}-\overline{x}^k\|+C_k(x^{k+1},\lambda^{k+1}),
\end{align}
where the second inequality is due to $\lambda^{k+1}\in\mathbb{R}_+^m$ and $G(\overline{x}^k,x^k,V^k,L^k)\in\mathbb{R}_-^m$.
Combining the above two inequalities with the inexactness condition \eqref{inexact2} yields
\[
 \frac{\mu_{\min}}{2}\|\overline{x}^k\!-\!x^k\|^2\le F_k(x^k)\!-\!F_k(x^{k+1})+\beta_{S}\|x^{k+1}\!-\!x^k\|\|x^{k+1}\!-\!\overline{x}^k\|+\frac{\beta_{C}}{2}\|x^{k+1}\!-\!x^k\|^2.
\]
By Assumption \ref{ass3} and Lemma \ref{lemma-bound} (i), there exists $\gamma>0$ such that $\|x^{k+1}\!-\!\overline{x}^k\|\le\gamma$. Along with the above inequality and the convexity of $F_k$, it follows
\[
  ({\mu_{\min}}/{2})\|\overline{x}^k\!-\!x^k\|^2\le \max_{\zeta\in\partial F_k(x^k)}(\|\zeta\|+\gamma\beta_{S})\|x^{k+1}-x^k\|+({\beta_{C}}/{2})\|x^{k+1}-x^k\|^2.
\]
Since $\partial F_k(x^k)=\xi^k+\partial\phi(x^k)$ with $\xi^k\in\partial g_0(x^k)$, $\max_{\zeta\in\partial F_k(x^k)}(\|\zeta\|+\gamma\beta_{S})$ is bounded by Remark \ref{remark-subdiff} (b). Then, from the above inequality and $\lim_{k\to\infty}\|x^{k+1}-x^k\|=0$ by Corollary \ref{corollary-objF}, we obtain $\lim_{k\to\infty}\|\overline{x}^{k}-x^k\|=0$. The proof is completed.
\end{proof}
\subsection{Subsequential convergence}\label{sec4.1}
This section aims at proving that every accumulation point of $\{x^k\}_{k\in\mathbb{N}}$ is a stationary point of \eqref{prob} by leveraging the parametric problem \eqref{para-prob} associated with $\mathbb{U}:=\mathbb{R}^n\times\mathbb{R}_{+}^m\times\mathbb{R}^{n\times m}\times\mathbb{R}^n\times\mathbb{S}_{+}$, where $\mathbb{S}_{+}$ signifies the set of all PSD linear mappings from $\mathbb{R}^n$ to $\mathbb{R}^n$, and the following
\begin{equation}\label{theta-def}
\theta(u,x):=g_0(s)+\langle\xi,x-s\rangle+\frac{1}{2}\langle x\!-\!s,\mathcal{Q}(x\!-\!s)\rangle\ {\rm and}\  H(u,x):=G(x,s,V,L)    
\end{equation} 
for $u=(s,L,V,\xi,\mathcal{Q})\in\mathbb{U}$. For each $k\in\mathbb{N}$, let $u^k\!:=(x^k,L^k,V^k,\xi^k,\mathcal{Q}_k)\in\mathbb{U}$. Then,
\[
\vartheta_k(\cdot)\!=\theta(u^k,\cdot),\ G(\cdot,x^k,V^k,L^k)=H(u^k,\cdot)\ \ {\rm and}\ \ \Gamma_k=\mathcal{S}(u^k)\quad\forall k\in\mathbb{N}.
\]
Obviously, the subproblem \eqref{subprob} is precisely the problem \eqref{para-prob} associated with $u=u^k$, and its multiplier set at $\overline{x}^k$ is $\mathcal{M}(u^k,\overline{x}^k)$, where $\mathcal{M}$ is the mapping defined in \eqref{para-multiplier}. 

Under Assumption \ref{ass3}, the sequence $\{(u^k,\overline{x}^{k})\}_{k\in\mathbb{N}}$ is bounded by Lemma \ref{lemma-bound}, so its cluster point set, denoted by $\Omega^*$, is nonempty and compact. Also, by Lemma \ref{lemma-xbark}, every $(u^*,\overline{x}^*)\in \Omega^*$ is of the form $(x^*,L^*,V^*,\xi^*,\mathcal{Q}^*,\overline{x}^*)$ with $\overline{x}^*=x^*$. Our analysis needs the following assumption on the partial BMP w.r.t. $x$ at all $(u^*,\overline{x}^*)\in \Omega^*$.
\begin{assumption}\label{ass4}
The constraint system of \eqref{para-prob} with $\theta$ and $H$ in \eqref{theta-def} satisfies the partial BMP w.r.t. $x$ at all $(u^*,\overline{x}^*)\in \Omega^*$, the set of cluster points of $\{(u^k,\overline{x}^{k})\}_{k\in\mathbb{N}}$.   
\end{assumption}
\begin{remark}\label{remark-ass4}
From \cite[Theorem 3.3 $\&$ Corollary 3.7]{Gfrerer17}, Assumption \ref{ass4} holds if for every $(u^*,\overline{x}^*)\in \Omega^*$ the mapping $\mathcal{H}_{u^*}(\cdot)\!:=H(u^*,\cdot)-\mathbb{R}_{-}^m=G(\cdot,\overline{x}^*,V^*,L^*)-\mathbb{R}_{-}^m$ is metrically regular at $(\overline{x}^*,0)$ or equivalently the partial MFCQ w.r.t. $x$ holds at $(u^*,\overline{x}^*)$ or equivalently ${\rm Ker}\,V^*\cap\mathcal{N}_{\mathbb{R}_{-}^m}(g(x^*))=\{0\}$. From the remark after Assumption \ref{ass2}, we learn that the MFCQ at $\overline{x}^*$ for \eqref{prob} implies the Slater's CQ for the convex constraint system $G(x,\overline{x}^*,V^*,L^*)\in\mathbb{R}_{-}^m$, while the latter implies the metric regularity of $\mathcal{H}_{u^*}$ at $(\overline{x}^*,0)$. Therefore, Assumption \ref{ass4} is weaker than the MFCQ for \eqref{prob} at all feasible points, a common CQ to analyze the convergence or iteration complexity of MBA methods (see, e.g., \cite[Assumption 2.3]{YuPongLv21}, \cite[Assumption 3]{Boob24} and \cite[Assumption 3]{Nabou24}). In addition, from \cite[Proposition 3.2]{Gfrerer17}, Assumption \ref{ass4} is also implied by the partial CRCQ w.r.t. $x$ at all $(u^*,\overline{x}^*)\in \Omega^*$, which is independent of the partial MFCQ w.r.t. $x$ at the corresponding $(u^*,\overline{x}^*)$ by virtue of the examples in \cite{Janin84}.
\end{remark}

\begin{proposition}\label{prop-lambark}
Under Assumptions \ref{ass2}-\ref{ass4}, there exist  $\widehat{k}\in\mathbb{N}$ and a constant $\beta_{\overline{\lambda}}>0$ such that for each $k\ge\widehat{k}$, a vector $\overline{\lambda}^k\in\mathcal{M}(u^k,\overline{x}^k)$ exists and $\|\overline{\lambda}^k\|\le\beta_{\overline{\lambda}}$.
\end{proposition}
\begin{proof}
 By Definition \ref{def-BMP}, for every $(u^{*},\overline{x}^{*})\in \Omega^*$, there exist $\varepsilon_{u^{*},\overline{x}^{*}}>0$ and $\kappa_{u^{*},\overline{x}^{*}}>0$ such that for all $(u,x)\in\mathbb{B}((u^*,\overline{x}^*),\varepsilon_{u^{*},\overline{x}^{*}})\cap[\mathbb{U}\times\mathcal{S}(u)]$ and $y\in\mathcal{N}_{\mathcal{S}(u)}(x)$,  
 \begin{equation}\label{BMP}  \Lambda(u,x,y)\cap\kappa_{u^{*},\overline{x}^{*}}\|y\|\mathbb{B}\ne\emptyset,
 \end{equation}
 where $\Lambda$ is the mapping defined in Definition \ref{def-BMP}. Since $\bigcup_{(u^*,\overline{x}^*)\in \Omega^*}\mathbb{B}^{\circ}((u^*,\overline{x}^*),\varepsilon_{u^{*},\overline{x}^{*}})$ is an open covering of the compact set $\Omega^*$, according to the Heine-Borel covering theorem, there exist $(u^{1,*},\overline{x}^{1,*}),\ldots,(u^{p,*},\overline{x}^{p,*})\in \Omega^*$ for some $p\in\mathbb{N}$ and $\varepsilon_{u^{1,*},\overline{x}^{1,*}}>0,\ldots,\varepsilon_{u^{p,*},\overline{x}^{p,*}}>0$ such that $ \Omega^*\subset\bigcup_{i=1}^p\mathbb{B}^{\circ}((u^{i,*},\overline{x}^{i,*}),\varepsilon_{u^{i,*},\overline{x}^{i,*}})$. Notice that the limit $\lim_{k\to\infty}{\rm dist}((u^k,\overline{x}^{k}),\Omega^*)=0$ holds. There exists an index $\widehat{k}\in\mathbb{N}$ such that for all $k\ge\widehat{k}$, 
 \(	(u^k,\overline{x}^k)\in\bigcup_{i=1}^p\mathbb{B}((u^{i,*},\overline{x}^{i,*}),\varepsilon_{u^{i,*},\overline{x}^{i,*}}).
 \)
 Write $\kappa\!:=\max_{i\in[p]}\kappa_{u^{i,*},\overline{x}^{i,*}}$ and $\varepsilon\!:=\min_{i\in[p]}\varepsilon_{u^{i,*},\overline{x}^{i,*}}$. For each $k\ge\widehat{k}$ (if necessary by increasing $\widehat{k}$), we have $(u^k,\overline{x}^k)\in\mathbb{B}((u^{i,*},\overline{x}^{i,*}),\varepsilon)\cap{\rm gph}\,\mathcal{S}$ for some $i\in[p]$. For each $k\in\mathbb{N}$, from the optimality of $\overline{x}^k$ to \eqref{para-prob} associated with $u=u^k$ and the expression of $\theta$ in \eqref{theta-def}, 
 \[ 0\in\nabla_{\!x}\theta(u^k,\overline{x}^k)+\partial\phi(\overline{x}^k)+\mathcal{N}_{\mathcal{S}(u^k)}(\overline{x}^k)=\partial F_k(\overline{x}^k)+\mathcal{N}_{\mathcal{S}(u^k)}(\overline{x}^k), 
 \]
 so there exists $-y^k\in\partial F_k(\overline{x}^k)$ such that $y^k\in\mathcal{N}_{\mathcal{S}(u^k)}(\overline{x}^k)$. From \eqref{BMP}, for each $k\ge\widehat{k}$, there exists $\overline{\lambda}^k\in\Lambda(u^k,\overline{x}^k,y^k)\subset\mathcal{M}(u^k,\overline{x}^k)$ such that $\|\overline{\lambda}^k\|\le\kappa\|y^k\|$, where the inclusion is due to \eqref{Mset-Lambda}. As $\partial F_k(\overline{x}^k)=\xi^k\!+\!\mathcal{Q}_k(\overline{x}^k\!-\!x^k)\!+\!\partial\phi(\overline{x}^k)$, we infer from Lemma \ref{lemma-bound} and Remark \ref{remark-subdiff} (b) that $\{y^k\}_{k\ge\widehat{k}}$ is bounded, so is $\{\overline{\lambda}^k\}_{k\ge\widehat{k}}$. 
\end{proof}

Proposition \ref{prop-lambark} identifies a bounded multiplier sequence from the set sequence $\{\mathcal{M}(u^k,\overline{x}^k)\}_{k\in\mathbb{N}}$, though the set sequence itself is not uniformly bounded. This is significantly different from the existing works on MBA-like methods, where such a bounded multiplier sequence was achieved under the MFCQ at all feasible points of \eqref{prob}. The latter ensures the uniform boundedness of the set sequence $\{\mathcal{M}(u^k,\overline{x}^k)\}_{k\in\mathbb{N}}$ and is stronger than Assumption \ref{ass4} by the discussion in Remark \ref{remark-ass4}. Now we are ready to apply Proposition \ref{prop-lambark} to prove the subsequential convergence of $\{x^k\}_{k\in\mathbb{N}}$.
\begin{theorem}\label{theorem-sconverge}
 Under Assumptions \ref{ass2}-\ref{ass4},  $\omega(x^0)\subset\Gamma^*$ and $F(x)\equiv\varpi^*$ on $\omega(x^0)$.
\end{theorem}
\begin{proof}
 Pick any $x^*\!\in\omega(x^0)$. There exists an index set $\mathcal{K}\subset\mathbb{N}$ such that $x^*=\lim_{\mathcal{K}\ni k\to\infty}x^k$, and  $\lim_{\mathcal{K}\ni k\to\infty}\overline{x}^k=x^*$ follows Lemma \ref{lemma-xbark}. In view of the boundedness of $\{u^k\}_{k\in\mathcal{K}}$ by Lemma \ref{lemma-bound}, if necessary by taking a subsequence, $\lim_{\mathcal{K}\ni k\to\infty}u^k=u^*$ for some $u^*=(x^*,L^*,V^*,\xi^*,\mathcal{Q}^*)\in\mathbb{U}$. From Proposition \ref{prop-lambark}, for each $k\ge\widehat{k}$, there exists $\overline{\lambda}^k\in\mathcal{M}(u^k,\overline{x}^k)$, and the sequence $\{\overline{\lambda}^k\}_{k\ge\widehat{k}}$ is bounded. If necessary by taking a subsequence, we assume $\lim_{\mathcal{K}\ni k\to\infty}\overline{\lambda}^k=\overline{\lambda}^*$. By the definition of $\mathcal{M}$ in \eqref{para-multiplier}, for each $k\ge\widehat{k}$, with $\overline{\lambda}^k\in\mathcal{N}_{\mathbb{R}_{-}^m}(H(u^k,\overline{x}^k))$, it holds 
 \begin{align*}\label{sconv-ineq1}
0&\in\nabla_{\!x}\theta(u^k,\overline{x}^k)+\partial\phi(\overline{x}^k)+\nabla\!_{x}H(u^k,\overline{x}^k)\overline{\lambda}^k\\
&=\xi^k+\mathcal{Q}_k(\overline{x}^k\!-\!x^k)+\partial\phi(\overline{x}^k)+V^k\overline{\lambda}^k+\langle L^k,\overline{\lambda}^k\rangle(\overline{x}^k\!-\!x^k).
 \end{align*}
Since $H(\cdot,\cdot)$ is continuous by \eqref{theta-def}, $\lim_{\mathcal{K}\ni k\to\infty}H(u^k,\overline{x}^k)=H(u^*,x^*)=g(x^*)$. Passing the limit $\mathcal{K}\ni k\to\infty$ to the inclusion $\overline{\lambda}^k\in\mathcal{N}_{\mathbb{R}_{-}^m}(H(u^k,\overline{x}^k))$ and the above inclusion, and using the osc property of $\mathcal{N}_{\mathbb{R}_{-}^m}(\cdot)$ and $\partial\phi$  leads to 
\[
  \overline{\lambda}^*\in\mathcal{N}_{\mathbb{R}_{-}^m}(g(x^*))\ \ {\rm and}\ \ 0\in\xi^*+\partial\phi(x^*)+V^*\overline{\lambda}^*.
\]
 Note that $\xi^*\in\partial g_0(x^*)$ and $V^*\in\partial g(x^*)$ because $\xi^k\in\partial g_0(x^k)$ and $V^k\in\partial g(x^k)$ for each $k\in\mathbb{N}$, and $\partial g_0$ and $\partial g$ are osc. The above equation means that $x^*$ is a stationary point of \eqref{prob}. The inclusion $\omega(x^0)\subset\Gamma^*$ then follows. Recall that $F$ is continuous relative to $\Gamma$ and $\{x^k\}_{k\in\mathbb{N}}\subset\Gamma$. Therefore, $F(x^*)=\lim_{\mathcal{K}\ni k\to\infty}F(x^k)= \varpi^*$. 
\end{proof} 

We see that the multiplier sequence $\{\lambda^k\}_{k\in\mathbb{N}}$ yielded by Algorithm \ref{iMBA} does not take a part in the proof of the subsequential convergence. Instead, it will join in the application of Proposition \ref{prop-ebound} to bound the distance of the inexact solution $x^{k+1}$ of \eqref{subprob} from its unique solution $\overline{x}^k$ by $\|x^{k+1}\!-x^k\|$ as follows. This result, as will be shown in Section \ref{sec4.2}, plays a crucial role in achieving the full convergence of $\{x^k\}_{k\in\mathbb{N}}$. 
\begin{proposition}\label{prop-eboundk}
 Under Assumptions \ref{ass2}-\ref{ass4}, there exists $\widehat{\gamma}>0$ such that 
 \[
 \|x^{k+1}-\overline{x}^{k}\|\le\widehat{\gamma}\|x^{k+1}-x^k\|\quad{\rm for\ all}\ k\ge\widehat{k}.
\]
\end{proposition}
\begin{proof}
 Fix any $k\ge\widehat{k}$. From \eqref{theta-def}, every component of $H(u^k,\cdot)$ is a convex function with $H(u^k,x^k)=G(x^k,x^k,L^k,V^k)\in\mathbb{R}_{-}^m$, and $\theta(u^k,\cdot)=\vartheta_k(\cdot)$ is strongly convex with modulus $\mu_{\min}$. Using Proposition \ref{prop-ebound} with $(\overline{u},\overline{x})=(u^k,\overline{x}^k)$ for $\overline{\lambda}=\overline{\lambda}^k$ of Proposition \ref{prop-lambark}, and $(x,v,\lambda)=(x^{k+1},v^{k+1},\lambda^{k+1})\in\mathbb{R}^n\times\partial\phi(x^{k+1})\times\mathbb{R}_+^m$ yields
 \begin{align}\label{keybound0}
 \mu_{\min}\|x^{k+1}\!-\!\overline{x}^k\|^2
 &\le\|x^{k+1}-\overline{x}^k\|\|\nabla\vartheta_k(x^{k+1})+v^{k+1}+\nabla_{\!x} G(x^{k+1},x^k,L^k,V^k)\lambda^{k+1}\|\nonumber\\		
 & +\|\overline{\lambda}^k\|_1\|[G(x^{k+1},x^k,L^k,V^k)]_{+}\|_{\infty}-\langle\lambda^{k+1},G(x^{k+1},x^k,L^k,V^k)\rangle\nonumber\\
 &\le S_k(x^{k+1}\!,v^{k+1}\!,\lambda^{k+1})\|x^{k+1}\!-\!\overline{x}^k\|+\max\{\|\overline{\lambda}^k\|_1,1\}C_k(x^{k+1}\!,\lambda^{k+1})\nonumber\\		
 &\stackrel{\eqref{inexact2}}{\le}\beta_{S}\|x^{k+1}\!-\!x^k\|\|x^{k+1}\!-\!\overline{x}^k\|\!+\!\beta_{C}\max\{\|\overline{\lambda}^k\|_1,1\}/2\|x^{k+1}-x^k\|^2.
 \end{align}
 where the second inequality is by the expressions of $S_k$ and $C_k$ in \eqref{resi-S}-\eqref{resi-C}. Since $\|\overline{\lambda}^k\|_1\le\sqrt{m}\|\overline{\lambda}^k\|\le\sqrt{m}\beta_{\overline{\lambda}}$ by Proposition \ref{prop-lambark}, the above inequality implies $\|x^{k+1}\!-\!\overline{x}^k\|\le \widehat{\gamma}\|x^{k+1}-x^k\|$ with $\widehat{\gamma}= (2\mu_{\min})^{-1}[\beta_{S}\!+\!\sqrt{\beta_{S}^2\!+\!2\mu_{\min}\beta_{C}\max\{\sqrt{m}\beta_{\overline{\lambda}},1\}}]$.
\end{proof}

\subsection{Full convergence and convergence rate}\label{sec4.2}

As well known, potential functions play a key part in the full convergence analysis of algorithms for nonconvex and nonsmooth optimization (see, i.e., \cite{YuPongLv21,Artacho25}), and their construction usually depends on the structure of subproblems. Write $\mathbb{Z}:=\mathbb{R}^n\times\mathbb{R}^n\times\mathbb{R}^{n\times m}\times\mathbb{R}_{+}^m\times\mathbb{R}^n$, and for each $k\in\mathbb{N}$, let $z^k:=(\overline{x}^{k},x^k,V^k,L^k,\xi^k)\in\mathbb{Z}$. For our problem, a natural choice is 
\[
 \Psi(z):=\phi(x)+\underbrace{g_0(s)+\langle\xi,x\!-\!s\rangle}_{G_0(x,s,\xi)}+\delta_{\mathbb{R}_-^m}(G(x,s,V,L))\quad\forall z=(x,s,V,L,\xi)\in\mathbb{Z}.
\]
Unfortunately, due to the nonsmoothness of the mapping $G$, the exact characterization for $\partial\Psi(z^k)$ is unavailable, so it is impossible to establish the relative error condition at $z^k$ for $\Psi$. Of course, one may consider replacing $\partial\Psi(z^k)$ with the regular one $\widehat{\partial}\Psi(z^k)$, but Algorithm \ref{iMBA} does not yield the elements of the latter set, and the relative error condition at $z^k$ for $\Psi$ is still inaccessible. To overcome this difficulty, we exploit the structure of upper-$\mathcal{C}^2$ functions on a compact convex set to construct functions $T_0$ and $T$, which possess more favorable properties than $G_0$ and $G$, respectively.

Under Assumption \ref{ass3}, the sequence $\{x^k\}_{k\in\mathbb{N}}$ is bounded, so we can find a compact convex set $D\subset\mathcal{O}$ such that $\{x^k\}_{k\in\mathbb{N}}\subset D$. By Lemma \ref{lemma-major} (ii), for each $i\in[m]_{+}$,
\begin{equation}\label{Egifun}
 g_i(x)=f_i(x)-h_i(x)\quad\forall x\in D,  
\end{equation}
where $f_i:\mathbb{R}^n\to\mathbb{R}$ and $h_i:\mathbb{R}^n\to\mathbb{R}$ are respectively a quadratic strongly convex function and a strongly convex function. For any $x\in\mathbb{R}^n$ and $Z\in\mathbb{R}^{n\times m}$, write
\[
  f(x):=(f_1(x),\ldots,f_m(x))^{\top}\ \ {\rm and}\ \ h^*(Z):=(h_1^*(Z_{1}),\ldots,h_m^*(Z_{m}))^{\top}.
\]
Define $T_0\!:\mathbb{R}^n\times\mathbb{R}^n\times\mathbb{R}^n\to\mathbb{R}$ and $T:\mathbb{R}^n\times\mathbb{R}^n\times\mathbb{R}^{n\times m}\times\mathbb{R}_{+}^m\to\mathbb{R}^m$ by
\begin{subequations}
\begin{align}\label{T0map}
 T_0(x,s,\xi)&:=\langle\xi,x\rangle+f_0(s)-\langle\nabla \!f_0(s),s\rangle+h_0^*(-\xi+\nabla\!f_0(s)),\\
 T(x,s,V,L)&:= V^{\top}x+f(s)-\mathcal{J}\!f(s)s+h^*(-V\!+\!\nabla\!f(s))+\frac{1}{2}\|x-s\|^2L.
 \label{Tmap}
\end{align}
\end{subequations}
Then, for each $k\in\mathbb{N}$, the following relations hold for the functions $T_0$ and $T$: 
\begin{subequations}
\begin{align}
 \label{GT0map} T_0(\overline{x}^k,x^k,\xi^k)&=g_0(x^k)+\langle\xi^k,\overline{x}^k\!-\!x^k\rangle,\\ T(\overline{x}^k,x^k,V^k,L^k)&=G(\overline{x}^k,x^k,V^k,L^k)\in\mathbb{R}_{-}^m.
 \label{GTmap}
\end{align}
\end{subequations}
Indeed, for each $k\in\mathbb{N}$, $\xi^k\in\partial g_0(x^k)=\nabla\!f_0(x^k)+\partial(-h_0)(x^k)$ and $V_{i}^k\in\partial g_i(x^k)=\nabla\!f_i(x^k)+\partial(-h_i)(x^k)$ for all $i\in[m]$, which in turn implies $\nabla\!f_0(x^k)-\xi^k\in\partial h_0(x^k)$ and $\nabla\!f_i(x^k)\!-\!V_{i}^k\in\partial h_i(x^k)$. Together with the convexity of $h_i$, it follows 
\begin{subequations}
\begin{align}\label{h0-equa}
 h_0(x^k)+h_0^*(\nabla\!f_0(x^k)-\xi^k)&=\langle\nabla\!f_0(x^k)-\xi^k,x^k\rangle,\\
 h_i(x^k)+h_i^*(\nabla\!f_i(x^k)\!-\!V_{i}^k)&=\langle\nabla\!f_i(x^k)\!-\!V_{i}^k,x^k\rangle\quad{\rm for}\ i\in[m].
 \label{hi-equa}
\end{align}    
\end{subequations}
Combining \eqref{h0-equa}-\eqref{hi-equa} with $g_i=f_i-h_i$ for each $i\in[m]_+$ leads to \eqref{GT0map} and 
\[
 \langle V_{i}^k,\overline{x}^k\rangle+f_i(x^k)-\langle\nabla\!f_i(x^k),x^k\rangle+h_i^*(\nabla\!f_i(x^k)\!-\!V_{i}^k)=g_i(x^k)+\langle V_{i}^k,\overline{x}^k-x^k\rangle.
\]
Adding $({L_i^k}/{2})\|\overline{x}^k-x^k\|^2$ to the both sides of the above equality yields \eqref{GTmap}. 

Inspired by the relations in \eqref{GT0map}-\eqref{GTmap}, we define the following function 
\begin{equation}\label{Phi-fun}
\Phi(z):=\phi(x)+T_0(x,s,\xi)+\delta_{\mathbb{R}_-^m}(T(x,s,V,L))\quad\forall z\in\mathbb{Z}.
\end{equation}
Then, as shown by Lemma \ref{lemma-potential} below, $\Phi(z^k)$ has the close relation with $F(x^{k})$. 
\begin{lemma}\label{lemma-potential}
 Under Assumptions \ref{ass2}-\ref{ass4}, there exists a constant $\rho_0\ge 0$ depending only on $g_0$ and $D$ such that for all $k\ge\widehat{k}$, with $\widetilde{c}:=[\beta_{\mathcal{Q}}(\widehat{\gamma}^2\!+\!1)\!+\!\beta_{S}\widehat{\gamma}+{\beta_{C}/2}\!+\!\rho_0]$, 
\[
 F(x^{k+1})-\widetilde{c}\,\|x^{k+1}-x^k\|^2\le \Phi(z^k)\le F(x^k).
\]
\end{lemma}
\begin{proof}
For each $k\ge\widehat{k}$, combining \eqref{sub-optval} with the inexactness condition \eqref{inexact2} and the second inequality of \eqref{inexact1} and using Proposition \ref{prop-eboundk}, we obtain 
\[
 F_k(x^{k+1})-F_k(\overline{x}^k)\le \beta_{F}\|x^{k+1}-x^k\|^2\ \ {\rm with}\ \beta_{F}:=\beta_{S}\widehat{\gamma}+{\beta_{C}/2}.
\]
 Together with the expression of $F_{k}$ and Lemma \ref{lemma-bound}, it follows
\begin{align}\label{potential-temp1}
 &\quad g_0(x^k)\!+\!\langle\xi^k,x^{k+1}\!-\!x^k\rangle\!+\!\phi(x^{k+1})\le F_{k}(\overline{x}^k)+{\beta_F}\|x^{k+1}-x^k\|^2\nonumber\\
&\le g_0(x^k)+\langle\xi^k,\overline{x}^k-x^k\rangle+\phi(\overline{x}^k)+({\beta_{\mathcal{Q}}}/{2})\|\overline{x}^k-x^k\|^2+\beta_F\|x^{k+1}-x^k\|^2.
\end{align}
From Lemma \ref{lemma-major} (ii) and the above \eqref{Egifun}, there exists a constant $\rho_{0}>0$ such that 
\begin{equation}\label{F0-ineq1}
 g_0(x^{k+1})\le g_0(x^k)+\langle\xi^k,x^{k+1}-x^k\rangle+\rho_{0}\|x^{k+1}-x^k\|^2\ \ {\rm for\ all}\ k\in\mathbb{N}.
\end{equation}
Now combining the expression of $F$ and the above inequalities \eqref{potential-temp1}-\eqref{F0-ineq1} yields
\begin{align*}
 F(x^{k+1})&\stackrel{\eqref{F0-ineq1}}{\le} g_0(x^k)+\langle \xi^k,x^{k+1}-x^k\rangle+{\rho_0}\|x^{k+1}-x^k\|^2+\phi(x^{k+1})\nonumber\\
 &\stackrel{\eqref{potential-temp1}}{\le} g_0(x^k)+\langle\xi^k,\overline{x}^k\!-\!x^k\rangle+\phi(\overline{x}^k)\!+\!\frac{\beta_{\mathcal{Q}}}{2}\|\overline{x}^k\!-\!x^k\|^2\!+\!(\beta_F\!+\!\rho_0)\|x^{k+1}\!-\!x^k\|^2\nonumber\\
&\le g_0(x^k)\!+\!\langle\xi^k,\overline{x}^k\!-\!x^k\rangle\!+\!\phi(\overline{x}^k)\!+\!\beta\!_{\mathcal{Q}}\|x^{k+1}\!-\!\overline{x}^k\|^2\!+\!(\beta\!_{\mathcal{Q}}\!+\!\beta_F\!+\!\rho_0)\|x^{k+1}\!-\!x^k\|^2\\
&\le g_0(x^k)+\langle\xi^k,\overline{x}^k\!-\!x^k\rangle+\phi(\overline{x}^k)+[\beta_{\mathcal{Q}}(\widehat{\gamma}^2\!+\!1)\!+\!\beta_F\!+\!\rho_0]\|x^{k+1}\!-\!x^k\|^2\\
&=\Phi(z^k)+[\beta_{\mathcal{Q}}(\widehat{\gamma}^2\!+\!1)\!+\!\beta_F\!+\!\rho_0]\|x^{k+1}\!-\!x^k\|^2
\end{align*}
where the fourth inequality is due to Proposition \ref{prop-eboundk}, and the equality is obtained by using \eqref{GT0map}-\eqref{GTmap} and the expression of $\Phi$. This proves the first inequality. The second inequality follows by noting that $F(x^k)=F_k(x^k)\ge F_k(\overline{x}^k)\ge F_k(\overline{x}^k)\!-\!\frac{1}{2}\langle\overline{x}^k\!-\!x^k,\mathcal{Q}_k(\overline{x}^k\!-\!x^k)\rangle =g_0(x^k)\!+\!\langle\xi^k,\overline{x}^k\!-\!x^k\rangle+\phi(\overline{x}^k)=T_0(\overline{x}^k,x^k,\xi^k)+\phi(\overline{x}^k)=\Phi(z^k)$.
\end{proof}

Motivated by Lemma \ref{lemma-potential}, we define the potential function $\Phi_{\widetilde{c}}:\mathbb{Z}\times\mathbb{R}^n\to\overline{\mathbb{R}}$ by 
\begin{equation}\label{Phic-def}
\Phi_{\widetilde{c}}(w):=\Phi(z)+\widetilde{c}\,\|d\|^2\quad\forall w=(z,d)\in\mathbb{Z}\times\mathbb{R}^n,
\end{equation}
where $\widetilde{c}$ is the same as in Lemma \ref{lemma-potential}. For each $k\in\mathbb{N}$, let $w^k\!:=(z^k,x^{k+1}\!-\!x^k)$. We next prove that $\Phi_{\widetilde{c}}$ keeps unchanged on the set of cluster points of $\{w^k\}_{k\in\mathbb{N}}$.  
\begin{proposition}\label{prop-ZWstar}
 Denote by $\mathcal{Z}^*$ the set of cluster points of $\{z^k\}_{k\in\mathbb{N}}$, and by $\mathcal{W}^*$ the set of cluster points of $\{w^k\}_{k\in\mathbb{N}}$. Then, under Assumptions \ref{ass2}-\ref{ass4}, 
 \begin{itemize}
 \item[(i)] the sets $\mathcal{Z}^*$ and $\mathcal{W}^*$ are nonempty and compact with $\mathcal{W}^*=\mathcal{Z}^*\times\{0\}$; 

 \item[(ii)] for each $z^*=(\overline{x}^*,x^*,V^*,L^*,\xi^*)\in\mathcal{Z}^*$, it holds $\overline{x}^*=x^*\in\omega(x^0)$ and  
 \begin{subequations}
 \begin{align}\label{relation1-sol}  &T_0(\overline{x}^*,x^*,\xi^*)\!=\!g_0(x^*),\,T(\overline{x}^*,x^*,V^*,L^*)\!=\!g(x^*)\in\mathbb{R}_-^m,\\
 &\nabla h_0(-\xi^*+\!\nabla\!f_0(x^*))=x^*= 
 \nabla h_i(-V_{i}^*+\!\nabla\!f_i(x^*))\ {\rm for}\ i\in[m]; 
 \label{relation2-sol}
 \end{align}
 \end{subequations}
 
 \item[(iii)] $\Phi(z^*)=\varpi^*=\Phi_{\widetilde{c}}(w^*)$ for all $z^*\in\mathcal{Z}^*$ and $w^*\in\mathcal{W}^*$;

 \item[(iv)] $\lim_{k\to\infty}\Phi(z^k)=\varpi^*=\lim_{k\to\infty}\Phi_{\widetilde{c}}(w^k)$.
 \end{itemize}
\end{proposition}
\begin{proof}
 Part (i) follows Lemma \ref{lemma-bound} and Corollary \ref{corollary-objF}. For parts (ii)-(iii), consider any $z^*=(\overline{x}^*,x^*,V^*,L^*,\xi^*)\in\mathcal{Z}^*$. Then there exists an index set $\mathcal{K}\subset\mathbb{N}$ such that $\lim_{\mathcal{K}\ni k\to\infty}z^k=z^*$, so  $x^*\in\omega(x^0)$. By Lemma \ref{lemma-xbark}, $\overline{x}^*=x^*$. From \eqref{GT0map}-\eqref{GTmap}, \eqref{h0-equa}-\eqref{hi-equa}, and the continuity of $T$ and $T_0$, we obtain \eqref{relation1-sol}-\eqref{relation2-sol}. From the expression of $\Phi$ and \eqref{relation1-sol}, we have $\Phi(z^*)=F(x^*)=\varpi^*$, where the last equality is by Theorem \ref{theorem-sconverge} and $x^*\in\omega(x^0)$, and for any $w^*\in\mathcal{W}^*$, $\Phi_{\widetilde{c}}(w^*)=\Phi(z^*)$ follows part (i). Part (iv) directly follows Lemma \ref{lemma-potential} and Corollary \ref{corollary-objF}. 
 \end{proof}

To achieve the full convergence of the sequence $\{x^k\}_{k\in\mathbb{N}}$, we also need to measure the approximate stationarity of its augmented sequence $\{w^k\}_{k\in\mathbb{N}}$ by $\|x^{k+1}\!-x^k\|$. 
\begin{proposition}\label{prop-relgap}
 Under Assumptions \ref{ass2}-\ref{ass4}, there is a constant $\gamma>0$ such that 
 \[
  {\rm dist}(0,\partial\Phi_{\widetilde{c}}(w^k))\le{\rm dist}(0,\widehat{\partial}\Phi_{\widetilde{c}}(w^k))\le\gamma\|x^{k+1}-x^k\|\quad{\rm for\ all}\ k\ge\widehat{k}.
 \]
\end{proposition}
\begin{proof}
 By \cite[Theorem 10.6]{RW98}, at any $z=(x,s,V,L,\xi)\in\mathbb{Z}$ for $T(x,s,V,L)\in\mathbb{R}_{-}^m$,  
 \begin{align}\label{Deltaz}  
 \partial\Phi(z)\supset\!\widehat{\partial}\Phi(z)&\supset\{\nabla\!_{x}T_0(x,s,\xi)\!+\!\partial\phi(x)\}\!\times\!\{\nabla\!_{s}T_0(x,s,\xi)\}\!\times\!\{0_{\mathbb{R}^{n\times m}}\}\!\times\{0_{\mathbb{R}^m}\}\nonumber\\
 &\times\!\{\nabla\!_{\xi}T_0(x,s,\xi)\}+\nabla T(x,s,V,L)\mathcal{N}_{\mathbb{R}_{-}^m}(T(x,s,V,L))\times\{0_{\mathbb{R}^m}\}:=\Delta(z).
 \end{align}
Fix any $k\ge\widehat{k}$. Let $\overline{\lambda}^k\!\in\mathcal{M}(u^k,\overline{x}^k)$ be the same as in Proposition \ref{prop-lambark} with $u^k\!=(x^k,V^k,L^k,\xi^k,\mathcal{Q}_k)$. By the expression of $\mathcal{M}$ in \eqref{para-multiplier} with $\theta$ and $H$ from \eqref{theta-def}, there exists $\overline{v}^k\in\partial\phi(\overline{x}^k)$ such that 
\begin{subequations}
 \begin{align}
 &\xi^k+\mathcal{Q}_k(\overline{x}^{k}\!-\!x^k)+\overline{v}^k+V^k\overline{\lambda}^{k}+\langle L^k,\overline{\lambda}^k\rangle(\overline{x}^{k}\!-\!x^k)=0, \label{subkkt1} \\
 &\overline{\lambda}^k\in\mathcal{N}_{\mathbb{R}_{-}^m}(G(\overline{x}^k,x^k,V^k,L^k))\stackrel{\eqref{GTmap}}{=}\mathcal{N}_{\mathbb{R}_{-}^m}(T(\overline{x}^k,x^k,V^k,L^k)). 
 \label{subkkt2}
\end{align}
\end{subequations} 
Combining \cite[Theorem 23.5]{Roc70} with \eqref{h0-equa}-\eqref{hi-equa}, it holds $\nabla h_0^*(\nabla\!f_0(x^k)\!-\!\xi^k)=x^k=\nabla h_i^*(\nabla\!f_i(x^k)\!-\!V_{i}^k)$ for each $i\in[m]$. Consequently, $\nabla\!_xT_0(\overline{x}^k,x^k,\xi^k)\!=\!\xi^k$, $\nabla\!_sT_0(\overline{x}^k,x^k,\xi^k)\!=\!0$, $\nabla\!_{\xi}T_0(\overline{x}^k,x^k,\xi^k)\!=\!\overline{x}^k-x^k$ and for any $\lambda\in\mathcal{N}_{\mathbb{R}_{-}^m}(T(\overline{x}^k\!,x^k\!,V^k\!,L^k))$, 
\begin{align*}
\nabla\!_{x}T(\overline{x}^k\!,x^k\!,V^k\!,L^k)\lambda\!=\!V^k\lambda\!+\!\langle L^k,\lambda\rangle(\overline{x}^k\!-\!x^k),\ 
\nabla\!_{s}T(\overline{x}^k\!,x^k\!,V^k\!,L^k)\lambda\!=\!\langle L^k,\lambda\rangle(x^k\!-\!\overline{x}^k),\\
\nabla\!_{V}T(\overline{x}^k\!,x^k\!,V^k\!,L^k)\lambda=(\overline{x}^k\!-\!x^k)\lambda^{\top},\ \nabla\!_{L}T(\overline{x}^k\!,x^k\!,V^k\!,L^k)\lambda\!=\!(1/2)\|\overline{x}^k\!-\!x^k\|^2\lambda.\qquad
\end{align*}
Using the above relations and the expression of $\Delta(z)$ in \eqref{Deltaz}, we immediately obtain
\[
\eta^k:=\begin{pmatrix}
 \xi^k+\overline{v}^k+V^k\overline{\lambda}^k+\langle L^k,\overline{\lambda}^k\rangle(\overline{x}^k-x^k)\\
 \langle L^k,\overline{\lambda}^k\rangle(x^k-\overline{x}^k)\\
 (\overline{x}^k-x^k)(\overline{\lambda}^k)^{\top}\\
 (1/2)\|\overline{x}^k-x^k\|^2\overline{\lambda}^k\\
 \overline{x}^k-x^k
\end{pmatrix}\in\Delta(z^k)\subset\widehat{\partial}\Phi(z^k).
\]
Notice that $\widehat{\partial}\Phi_{\widetilde{c}}(w^k)=\widehat{\partial}\Phi(z^k)\times\{2\widetilde{c}(x^{k+1}\!-\!x^k)\}$ for each $k\ge\widehat{k}$. It suffices to claim that there exists $\gamma>0$ such that $\|\eta^k\|\le \gamma\|x^{k+1}-x^k\|$ for all $k\ge\widehat{k}$. Indeed, for each $k\ge\widehat{k}$, from equation \eqref{subkkt1}, Lemma \ref{lemma-bound} and Proposition \ref{prop-lambark}, it follows
\[
\|\eta^k\|\le\big(\beta_{\mathcal{Q}}+\beta_L\beta_{\overline{\lambda}}+\beta_{\overline{\lambda}}+1\big)\|\overline{x}^k-x^k\|+(1/2)\beta_{\overline{\lambda}}\|\overline{x}^k-x^k\|^2.
\]
Recall that $\lim_{k\to\infty}\|\overline{x}^{k}-x^k\|=0$ by Lemma \ref{lemma-xbark}, while $\|\overline{x}^{k}-x^k\|\le\|\overline{x}^{k}-x^{k+1}\|+\|x^{k+1}-x^k\|\le(\widehat{\gamma}+1)\|x^{k+1}-x^k\|$ for each $k\ge\widehat{k}$ by Proposition \ref{prop-eboundk}.  Thus, from the above inequality, there exists $\gamma>0$ such that $\|\eta^k\|\le\gamma\|x^{k+1}-x^k\|$ for all $k\ge\widehat{k}$. 
\end{proof}

Proposition \ref{prop-relgap} provides a relatively inexact optimality condition for minimizing $\Phi_{\widetilde{c}}$ with the sequence $\{w^k\}_{k\in\mathbb{N}}$. However, we cannot apply directly the recipe developed in \cite{Attouch13,Bolte14} to get the convergence of $\{w^k\}_{k\in\mathbb{N}}$, since $\{\Phi_{\widetilde{c}}(w^k)\}_{k\in\mathbb{N}}$ lacks the sufficient decrease. From the expression of $\Phi_{\widetilde{c}}$ in \eqref{Phic-def} and Lemma \ref{lemma-potential}, for all $k\ge\widehat{k}$, 
\begin{equation*}
 \Phi_{\widetilde{c}}(w^k)-\varpi^*
  =\Phi(z^k)+\widetilde{c}\|x^{k+1}-x^k\|^2-\varpi^*\ge F(x^{k+1})-\varpi^*>0.
\end{equation*}  
Based on this, we can establish the full convergence of $\{x^k\}_{k\in\mathbb{N}}$ by combining the KL inequality on $\Phi_{\widetilde{c}}$ and the decrease of $\{F(x^k)\}_{k\in\mathbb{N}}$. 
\begin{theorem}\label{globalconv}
 Under Assumptions \ref{ass2}-\ref{ass4}, if $\Phi_{\widetilde{c}}$ has the KL property on $\mathcal{W}^*$, then $\sum_{k=0}^{\infty}\|x^{k+1}\!-x^k\|<\infty$, so the sequence $\{x^k\}_{k\in\mathbb{N}}$ converges to some $x^*\in\omega(x^0)\subset\Gamma^*$.
\end{theorem}
\begin{proof}
 If there exists some $k_0\in\mathbb{N}$ such that $F(x^{k_0})=F(x^{k_0+1})$, then $x^{k_0+1}=x^{k_0}$ follows Corollary \ref{corollary-objF}. By Remark \ref{remark-alg} (c), Algorithm \ref{iMBA} finds a stationary point within a finite number of steps. Hence, it suffices to consider that $F(x^k)>F(x^{k+1})>\varpi^*$ for all $k\in\mathbb{N}$. From the expression of $\Phi_{\widetilde{c}}$ in \eqref{Phic-def} and Lemma \ref{lemma-potential}, for each $k\ge\widehat{k}$, it holds that
 \begin{align}\label{relationPhi}
	\Phi_{\widetilde{c}}(w^k)-\varpi^*
	=\Phi(z^k)+\widetilde{c}\|x^{k+1}-x^k\|^2-\varpi^*\ge F(x^{k+1})-\varpi^*>0.
 \end{align}  
 By Proposition \ref{prop-ZWstar}, the set $\mathcal{W}^*$ is nonempty and compact, and $\Phi_{\widetilde{c}}(w)=\varpi^*$ for all $w\in \mathcal{W}^*$. From the KL property of $\Phi_{\widetilde{c}}$ on $\mathcal{W}^*$ and \cite[Lemma 6]{Bolte14}, there exist $\varepsilon>0,\eta>0$ and $\varphi\in\Upsilon_{\!\eta}$ such that for all $w\in[\varpi^*<\Phi_{\widetilde{c}}<\varpi^*+\eta]\cap\mathfrak{B}(\mathcal{W}^*,\varepsilon)$ with $\mathfrak{B}(\mathcal{W}^*,\varepsilon)\!:=\big\{w\in\mathbb{W}\,|\,{\rm dist}(w,\mathcal{W}^*)\le\varepsilon\big\}$, $\varphi'(\Phi_{\widetilde{c}}(w)-\varpi^*){\rm dist}(0,\partial\Phi_{\widetilde{c}}(w))\ge 1$.
 Since $\lim_{k\to\infty}\Phi_{\widetilde{c}}(w^k)=\varpi^*$ by Proposition \ref{prop-ZWstar} (iv) and $\lim_{k\to\infty}{\rm dist}(w^k,\mathcal{W}^*)=0$, if necessary by increasing $\widehat{k}$, it holds that $w^k\in[\varpi^*<\Phi_{\widetilde{c}}<\varpi^*\!+\!\eta]\cap\mathfrak{B}(\mathcal{W}^*,\varepsilon)$ for all $k>\widehat{k}$, so 
 \[
	 \varphi'\big(\Phi_{\widetilde{c}}(w^k)\!-\!\varpi^*\big){\rm dist}(0,\partial\Phi_{\widetilde{c}}(w^k))\geq 1. 
 \]
 Together with Proposition \ref{prop-relgap} and $\varphi\in\Upsilon_{\eta}$, it follows that for all $k\ge\widehat{k}$, 
 \begin{equation*}
	\gamma\varphi'\big(\Phi_{\widetilde{c}}(w^{k-1})-\varpi^*\big)\|x^k-x^{k-1}\|\ge 1.
 \end{equation*}
 Since $\varphi'$ is nonincreasing on $(0,\eta)$, the above inequality along with \eqref{relationPhi} implies 
 \begin{equation}\label{equa1-later}
	\varphi'(F(x^k)-\varpi^*)\ge\varphi'(\Phi_{\widetilde{c}}(w^{k-1})-\varpi^*)
	\ge\frac{1}{\gamma\|x^k-x^{k-1}\|}\quad \forall k\ge\widehat{k}.
 \end{equation}
 Now using \eqref{equa1-later} and the sufficient decrease of $\{F(x^k)\}_{k\in\mathbb{N}}$ in \eqref{F-decrease} and following the same arguments as those in \cite[Theorem 1 (i)]{Bolte14} leads to the conclusion. For completeness, we include the details. Invoking the concavity of $\varphi$ and \eqref{F-decrease} leads to 
 \begin{align*}
	\Delta_{k,k+1}&:=\varphi(F(x^k)-\varpi^*)-\varphi(F(x^{k+1})-\varpi^*)
	\ge\varphi'(F(x^k)-\varpi^*)(F(x^{k})\!-\!F(x^{k+1}))\\
	&\stackrel{\eqref{equa1-later}}{\ge}\frac{F(x^{k})\!-\!F(x^{k+1})}{\gamma\|x^k-x^{k-1}\|}
	\ge\frac{\alpha\|x^{k+1}-x^k\|^2}{2\gamma\|x^k-x^{k-1}\|}\quad\forall k\ge\widehat{k}.
 \end{align*}
 Then, $\|x^{k+1}\!-\!x^k\|\le\sqrt{{2\gamma}\alpha^{-1}\Delta_{k,k+1}\|x^k\!-\!x^{k-1}\|}$ for all $k\ge\widehat{k}$. Along with $2\sqrt{ab}\le a+b$ for $a\ge 0,b\ge 0$, we have $2\|x^{k+1}\!-\!x^k\|\le 2\alpha^{-1}\gamma\Delta_{k,k+1}+\|x^k\!-\!x^{k-1}\|$ for all $k\ge\widehat{k}$. Summing this inequality from any $k\ge\widehat{k}$ to any $l>k$ yields that
 \begin{align*}
	2{\textstyle\sum_{i=k}^l}\|x^{i+1}-x^i\|&\le{\textstyle\sum_{i=k}^l}\|x^i-x^{i-1}\|+2\alpha^{-1}\gamma{\textstyle\sum_{i=k}^l}\Delta_{i,i+1}\\
	&\le{\textstyle\sum_{i=k}^l}\|x^{i+1}-x^i\|+\|x^k-x^{k-1}\|+2\alpha^{-1}\gamma\varphi\big(F(x^{k})-\varpi^*\big),
 \end{align*}
 where the second inequality is by the nonnegativity of $\varphi$. Thus, for any $k\ge\widehat{k}$, 
 \begin{equation}\label{temp-ratek}
	{\textstyle\sum_{i=k}^l}\|x^{i+1}-x^i\|\le\|x^k-x^{k-1}\|+2\alpha^{-1}\gamma\varphi(F(x^{k})-\varpi^*).
 \end{equation}
 Passing the limit $l\to\infty$ to this inequality leads to $\sum_{k=0}^{\infty}\|x^{k+1}\!-x^k\|<\infty$. 
\end{proof}

Recall that $f_i$ for $i\in[m]_{+}$ are strongly convex quadratic functions. When $g_i$ for $i\in[m]_{+}$ and $\phi$ are definable in the same o-minimal structure over the real field $(\mathbb{R},+,\cdot)$, from \cite[Section 4]{Attouch10}, $h_i$ for $i\in[m]_{+}$ and $\phi$ are definable in this o-minimal structure, so are $h_i^*$ for $i\in[m]_{+}$ and $\phi$. Together with   
the expressions of $\Phi$ and $\Phi_{\widetilde{c}}$, now  $\Phi_{\widetilde{c}}$ is definable in this o-minimal structure, so is a KL function. 

By \cite[Theorem 3.3]{LiPong18}, if $\Phi$ satisfies the KL property of exponent $q\in[1/2,1)$ at $\overline{z}\in\mathcal{Z}^*$, then $\Phi_{\widetilde{c}}$ has the KL property of exponent $q\in[1/2,1)$ at $(\overline{z},0)\in\mathcal{W}^*$ since the strongly convex function $\mathbb{R}^n\ni d\mapsto \widetilde{c}\|d\|^2$ is a KL function of exponent $1/2$. Then, following the proof of \cite[Theorem 2]{Attouch09}, we can get the following convergence rate result. 
\begin{theorem}\label{KL-rate}
 Under Assumptions \ref{ass2}-\ref{ass4}, if $\Phi$ has the KL property of exponent $q\in[1/2,1)$ on the set $\mathcal{Z}^*$, then the following assertions hold: 
 \begin{itemize}
 \item[(i)] for $q=\frac{1}{2}$, there are $\widehat{a}_1>0$ and $\varrho,\widehat{\varrho}\in(0,1)$ such that for all $k$ large enough, 
\[
  F(x^{k+1})-\varpi^*\le \varrho(F(x^k)-\varpi^*)\ \ {\rm and}\ \ \|x^k-x^*\|\le \widehat{a}_1\widehat{\varrho}^{k};
 \]

 \item[(ii)] for $q\in(\frac{1}{2},1)$, there exist $\widehat{a}_2>0$ and $\widehat{a}_3>0$ such that for large enough $k$, 
\[
 F(x^k)-\varpi^*\le \widehat{a}_2 k^{-\frac{1}{2q-1}}\ \ {\rm and}\ \ \|x^k-x^*\|\le \widehat{a}_3 k^{-\frac{1-q}{2q-1}}.
\]
\end{itemize}
\end{theorem}
\begin{proof}
 From the discussion before the theorem, $\Phi_{\widetilde{c}}$ has the KL property of exponent $q\in[1/2,1)$ on the set $\mathcal{W}^*$.
 According to Theorem \ref{globalconv}, the sequence $\{x^k\}_{k\in\mathbb{N}}$ converges to $x^*$. To prove items (i)-(ii), for each $k\in\mathbb{N}$, write $\Delta_k\!:=\!\sum_{j=k}^{\infty}\|x^{j+1}\!-\!x^j\|$ and $\varpi_k:=F(x^k)-\varpi^*$. From the proof of Theorem \ref{globalconv}, it suffices to consider that $F(x^k)>F(x^{k+1})>\varpi^*$ for all $k\in\mathbb{N}$. Note that \eqref{equa1-later} holds for $\mathbb{R}_{+}\ni t\mapsto\varphi(t)=ct^{1-q}\ (c>0)$, i.e., 
 \begin{equation}\label{ineq-F}
 (F(x^{k})\!-\!\varpi^*)^{q}\le(\Phi_{\widetilde{c}}(w^{k-1})-\varpi^*)^{q} 
 \le\gamma c(1-q)\|x^k-x^{k-1}\|\quad\forall k\ge\widehat{k}.
 \end{equation}
 Combining \eqref{ineq-F} with \eqref{F-decrease} shows that the following inequality holds for all $k\ge\widehat{k}$:
 \begin{equation}\label{recursion1-Vk}
 \varpi_k^{2q}\le[\gamma c(1-q)]^2\|x^{k}-x^{k-1}\|^2\le c_1(\varpi_{k-1}-\varpi_k)\ \ {\rm with}\ c_1=2\alpha^{-1}[\gamma c(1-q)]^2.
 \end{equation}
 In addition, by passing the limit $l\to\infty$ to \eqref{temp-ratek} and using $\varphi(t)=ct^{1-q}$, it follows
 \[
 \Delta_k\le \|x^k\!-\!x^{k-1}\|+2\alpha^{-1}\gamma c\big(F(x^k)-\varpi^*\big)^{1-q}\quad\forall k\ge\widehat{k}_1.
 \]
 Recall that $\lim_{k\to\infty}\|x^{k+1}-x^k\|=0$ and $\frac{1-q}{q}\le1$. If necessary by increasing $\widehat{k}_1$, we have $\|x^k\!-\!x^{k-1}\|\le \|x^k\!-\!x^{k-1}\|^{\frac{1-q}{q}}$. Then, the above inequality implies that for all $k\ge\widehat{k}_1$, 
 \begin{align}\label{recursion1-Deltak}
 \Delta_k&\le \|x^k\!-\!x^{k-1}\|^{\frac{1-q}{q}}+2\alpha^{-1}\gamma c\big(F(x^k)-\varpi^*\big)^{1-q}\nonumber\\
 &\stackrel{\eqref{ineq-F}}{\le}\|x^k-x^{k-1}\|^{\frac{1-q}{q}}+2\alpha^{-1}(\gamma c)^{\frac{1}{q}}(1\!-\!q)^{\frac{1-q}{q}}\|x^k-x^{k-1}\|^{\frac{1-q}{q}}\nonumber\\ 
 &=c_2(\Delta_{k-1}-\Delta_k)^{\frac{1-q}{q}}\ \ {\rm with}\ c_2=1\!+2\alpha^{-1}(\gamma c)^{\frac{1}{q}}(1\!-\!q)^{\frac{1-q}{q}}.
 \end{align}
 When $q=1/2$, \eqref{recursion1-Vk} implies that $\varpi_k\le\frac{c_1}{1+c_1}\varpi^{k-1}$ for all $k\ge\widehat{k}$, which is precisely the first inequality in item (i) with $\varrho=\frac{c_1}{1+c_1}$. 
 With \eqref{recursion1-Vk}-\eqref{recursion1-Deltak}, using the same arguments as those for \cite[Theorem 2] {Attouch09} yields the rest of the result. 
\end{proof}
\subsection{KL property of $\Phi$ with exponent $q\in[1/2,1)$}\label{sec4.3}
Among nonconvex and nonsmooth functions, the KL functions of exponent $q\in[1/2,1)$ are few and far between, except for those with some special structures (see, e.g., \cite{LiPong18,WuPanBi21,YuLiPong21}). Then, for a complex nonconvex and nonsmooth function, it is important to provide a reasonable condition to ensure its KL property of exponent $q\in[1/2,1)$ at a given critical point. This section focuses on this topic for the function $\Phi$. To that end, define
\begin{equation} \psi(x,t,\omega):=\phi(x)+t+\delta_{\mathbb{R}_{-}^m}(\omega)\quad{\rm for}\ (x,t,\omega)\in\mathbb{R}^n\times\mathbb{R}\times\mathbb{R}^m.
\end{equation}
From \eqref{Phi-fun},  $\Phi(z)=\psi(x,T_0(x,s,\xi),T(x,s,V,L))$ for $z\in\mathbb{Z}$. Note that $\psi$ is a KL function of exponent $q\in[1/2,1)$ if and only if $\phi$ is a KL function of exponent $q\in[1/2,1)$. Based on this fact and the composite structure of $\Phi$, we provide a checkable condition for $\Phi$ to have the KL property of exponent $q\in[1/2,1)$ at its critical points. 
\begin{proposition}\label{prop-KLexp}
 Fix any $\overline{z}=\!(\overline{x},\overline{s},\overline{V},\overline{L},\overline{\xi})\in\mathcal{Z}^*$. If $\phi$ has the KL property of exponent $q\in[1/2,1)$ at $\overline{x}$, then $\Phi$ has the KL property of exponent $q$ at $\overline{z}$ when
 \begin{equation}\label{cond-KL}
 (r,v,\lambda)\in [0,1]\times r\partial\phi(\overline{x})\times\mathcal{N}_{\mathbb{R}_{-}^m}(g(\overline{x})),\ r\overline{\xi}+v+\overline{V}\lambda=0 \ \Longrightarrow\ (r,\lambda)=0.
 \end{equation}
\end{proposition}
\begin{proof}
 From the expression of $\Delta(z)$ in \eqref{Deltaz}, it is not difficult to check that $0\in\Delta(\overline{z})\subset\partial\Phi(\overline{z})$, so $\overline{z}$ is a critical point of $\Phi$. Define the following multifunction 
 \begin{equation}
 \mathcal{T}(x,s,V,L):=T(x,s,V,L)-\mathbb{R}_{-}^m\ \ {\rm for}\  (x,s,V,L)\in\mathbb{R}^n\times\mathbb{R}^n\times\mathbb{R}^{n\times m}\times\mathbb{R}_{+}^m. 
 \end{equation}
 Then, the mapping $\mathcal{T}$ is metrically regular at $((\overline{x},\overline{s},\overline{V},\overline{L}),0_{\mathbb{R}^m})$ if and only if ${\rm Ker}\,\overline{V}\cap\mathcal{N}_{\mathbb{R}_{-}^m}(g(\overline{x}))=\{0_{\mathbb{R}^m}\}$, and the latter is implied by the condition \eqref{cond-KL}. From the robustness of metric regularity, there exists $\varepsilon>0$ such that $\mathcal{T}$ is metrically regular at any $((x,s,V,L),0_{\mathbb{R}^m})$ with $(x,s,V,L)\in\mathbb{B}((\overline{x},\overline{s},\overline{V},\overline{L}),\varepsilon)\cap T^{-1}(\mathbb{R}_{-}^m)$. 
 
 Suppose on the contrary that $\Phi$ does not satisfy the KL property of exponent $q$ at $\overline{z}$. According to  Definition \ref{KL-def}, there exists a sequence $z^l:=(x^l,s^l,V^l,L^l,\xi^l)\to\overline{z}$ with $\Phi(\overline{z})<\Phi(z^l)<\Phi(\overline{z})+\frac{1}{l}$ such that for each $l\in\mathbb{N}$, 
 \[
 {\rm dist}(0,\partial\Phi(z^l))<\frac{1}{l}(\Phi(z^l)-\Phi(\overline{z}))^{q}=\frac{1}{l}\big(\psi(x^l,t^l,\omega^l)-\psi(\overline{x},\overline{t},\overline{\omega})\big)^{q}, 
 \]
 where $t^l=T_0(x^l,s^l,\xi^l)$ and $\omega^l=T(x^l,s^l,V^l,L^l)$, and $\overline{t}=T_0(\overline{x},\overline{s},\overline{\xi})=g_0(\overline{x})$ and $\overline{\omega}=T(\overline{x},\overline{s},\overline{V},\overline{L})=g(\overline{x})$. Obviously, $\Phi(\overline{z})=\psi(\overline{x},\overline{t},\overline{\omega})=F(\overline{x})=\varpi^*$. Moreover, for each $l\in\mathbb{N}$, $z^l\in{\rm dom}\,\Phi$, so that $(x^l,s^l,V^l,L^l)\in T^{-1}(\mathbb{R}_{-}^m)$ and $\Phi(z^{l})=\psi(x^{l},t^{l},\omega^{l})$. Along with $z^l\to\overline{z}$, there exists $\widetilde{l}\in\mathbb{N}$ such that $\mathcal{T}$ is metrically regular at each $((x^l,s^l,V^l,L^l),0_{\mathbb{R}^m})$ for $l\ge\widetilde{l}$. Note that the inclusions in \eqref{Deltaz} become equalities at $z^l$ if the mapping $\mathcal{T}$ is subregular at $((x^l,s^l,V^l,L^l),0_{\mathbb{R}^m})$ since $\partial\Phi(z^l)\subset\Delta(z^l)$ by \cite[Page 211]{Ioffe08-calm}. Then, for each $l\ge\widetilde{l}$, $\partial\Phi(z^l)=\Delta(z^l)$, so there exist $v^l\in\partial\phi(x^l),y^{0,l}:=\nabla h_0^*(-\xi^l\!+\!\nabla\!f_0(s^l)),\lambda^l\in\mathcal{N}_{\mathbb{R}_{-}^m}(\omega^l)$ and $y^{i,l}:=\nabla h_i^*(-V_{i}^l+\nabla\!f_i(s^l))$ such that
\begin{align}\label{kl1}
 &\left\|\begin{pmatrix}
 \xi^l+v^l+V^l\lambda^l+\langle L^l,\lambda^l\rangle(x^l-s^l)\\
 \nabla^2\!f_0(s^l)(y^{0,l}\!-\!s^l)+\sum_{i=1}^m\lambda_i^l\nabla^2\!f_i(s^l)(s^l\!-\!y^{i,l})+\langle L^l,\lambda^l\rangle(s^l\!-\!x^l)\\
 [\lambda_1^l(x^l\!-\!y^{1,l})\ \cdots\ \lambda_m^l(x^l\!-\!y^{m,l})]\\
 \frac{1}{2}\|x^l-s^l\|^2\lambda^l\\
  x^l-y^{0,l}
 \end{pmatrix}\right\|\nonumber\\
 &<\frac{1}{l}\big(\psi(x^l,t^l,\omega^l)-\varpi^*\big)^{q}.
\end{align}
 Since $\psi$ has the KL property of exponent $q$ at $(\overline{x},\overline{t},\overline{\omega})$, there exist $\delta>0,\eta>0$ and $c>0$ such that for all $(x,t,\omega)\in\mathbb{B}((\overline{x},\overline{t},\overline{\omega}),\delta)\cap[\varpi^*<\psi<\varpi^*+\eta]$,
 \begin{equation*}
 {\rm dist}(0,\partial\psi(x,t,\omega))\ge c(\psi(x,t,\omega)-\varpi^*)^{q}.
 \end{equation*}
 From $z^l\to\overline{z}$ and the expression of $T$ in \eqref{Tmap}, we have $\omega^l\to\overline{\omega}$. If necessary by increasing $\widetilde{l}$, for all $l\ge\widetilde{l}$, $(x^l,t^l,\omega^l)\in\mathbb{B}((\overline{x},\overline{t},\overline{\omega}),\delta)\cap[\varpi^*\!<\psi<\varpi^*\!+\!\eta]$. Then  
 \begin{equation}\label{kl3}
  c(\psi(x^l,t^l,\omega^l)-\varpi^*)^q\le{\rm dist}(0,\partial\psi(x^l,t^l,\omega^l))
  \le\|(v^l,1,\lambda^l)\|:=\varsigma_l,
 \end{equation}
 where the second inequality is due to $(v^l,1,\lambda^l)\in\partial \psi(x^l,t^l,\omega^l)$ for each $l\in\mathbb{N}$. Note that the sequence $\{v^l\}_{l\in\mathbb{N}}$ is bounded. We claim that the sequence $\{\lambda^l\}_{l\in\mathbb{N}}$ is bounded. If not, we  assume that $\lim_{l\to\infty}\frac{\lambda^l}{\|\lambda^l\|}=\widetilde{\lambda}$ with $\|\widetilde{\lambda}\|=1$ (if necessary by taking a subsequence). Obviously, $\frac{\lambda^l}{\|\lambda^l\|}\in\mathcal{N}_{\mathbb{R}_{-}^m}(\omega^l)$ for each $l\in\mathbb{N}$, which by the outer semicontinuity of $\mathcal{N}_{\mathbb{R}_{-}^m}(\cdot)$ implies $\widetilde{\lambda}\in\mathcal{N}_{\mathbb{R}_{-}^m}(g(\overline{x}))$. On the other hand, dividing the both sides of \eqref{kl1} by $\|\lambda^l\|$ and passing the limit $l\to\infty$ results in $\overline{V}\widetilde{\lambda}=0$. Then, the metric regularity of $\mathcal{T}$ at $((\overline{x},\overline{s},\overline{V},\overline{L}),0_{\mathbb{R}^m})$ implies that $\widetilde{\lambda}=0$, a contradiction to $\|\widetilde{\lambda}\|=1$. Thus, $\{\lambda^l\}_{l\in\mathbb{N}}$ is bounded.  Along with the definition of $\varsigma_l$, we deduce $1\le\varsigma_l\le \overline{\varsigma}$ for some $\overline{\varsigma}\ge 1$. 
 For each $l\ge\widetilde{l}$, write $(\widetilde{v}^l,\widetilde{\lambda}^l):=\varsigma_l^{-1}(v^l,\lambda^l)$. Obviously, $\|(\widetilde{v}^l,\varsigma_l^{-1},\widetilde{\lambda}^l)\|=1$. There exists an index set $\mathcal{K}\subset\mathbb{N}$ such that $\lim_{\mathcal{K}\ni l\to\infty}(\widetilde{v}^l,\varsigma_l^{-1},\widetilde{\lambda}^l)=(\widetilde{v},\widetilde{r},\widetilde{\lambda})$ with $\overline{\varsigma}^{-1}\le \widetilde{r}\le 1$. Together with the above \eqref{kl1} and \eqref{kl3}, for each $l\ge\widetilde{l}$, it holds  
\begin{equation*}
 \left\|\begin{pmatrix} \varsigma_l^{-1}\xi^l+\widetilde{v}^l+V^l\widetilde{\lambda}^l+\langle L^l,\widetilde{\lambda}^l\rangle(x^l-s^l)\\
 \varsigma_l^{-1}\nabla^2\!f_0(s^l)(y^{0,l}\!-\!s^l)+\sum_{i=1}^m\widetilde{\lambda}_i^l\nabla^2\!f_i(s^l)(s^l\!-\!y^{i,l})+\langle L^l,\widetilde{\lambda}^l\rangle(s^l\!-\!x^l)\\
 [\widetilde{\lambda}_1^l(x^l\!-\!y^{1,l})\ \cdots\ \widetilde{\lambda}_m^l(x^l\!-\!y^{m,l})]\\
 \frac{1}{2}\|x^l-s^l\|^2\widetilde{\lambda}^l\\
 \varsigma_l^{-1}(x^l-y^{0,l})
\end{pmatrix}\right\|\le\frac{1}{cl}.
\end{equation*}
Recall that $z^l\to\overline{z}\in \mathcal{Z}^*$. From the definition of $y^{i,l}$ for $i\in[m]_{+}$ and Proposition \ref{prop-ZWstar} (ii), it follows  $\lim_{l\to\infty}y^{0,l}=\overline{x}$ and $\lim_{l\to\infty}y^{i,l}=\overline{x}$ for all $i\in[m]$. Now passing the limit $\mathcal{K}\ni l\to\infty$ to the above inequality results in $\widetilde{v}+\widetilde{r}\overline{\xi}+\overline{V}\widetilde{\lambda}=0$. On the other hand, recall that $v^l\in\partial\phi(x^l)$ and $\lambda^l\in\mathcal{N}_{\mathbb{R}_-^m}(\omega^l)$, for each $l\ge\widetilde{l}$, it holds $(\widetilde{v}^l,\widetilde{\lambda}^l)\in\varsigma_l^{-1}\partial\phi(x^l)\times\mathcal{N}_{\mathbb{R}_-^m}(\omega^l)$.
Passing the limit $\mathcal{K}\ni l\to\infty$ to the inclusion and using the osc property of the mapping $\mathcal{N}_{\mathbb{R}_-^m}(\cdot)$ results in
\(
(\widetilde{v},\widetilde{\lambda})\in \widetilde{r}\,\partial\phi(\overline{x})\times\mathcal{N}_{\mathbb{R}_-^m}(g(\overline{x})).
\)
The above two equations yield a contradiction to \eqref{cond-KL} since $\widetilde{r}\ge\overline{\varsigma}^{-1}>0$.
\end{proof}
\begin{remark}
 {\bf(a)} Proposition \ref{prop-KLexp} provides a condition to check the KL property of $\Phi$ with exponent $q\in[1/2,1)$ at the given critical point $\overline{z}$. Note that $\Phi=\psi\circ\Xi$ with $\Xi(z):=(x,T_0(x,s,\xi),T(x,s,V,L))$ for $z\in\mathbb{Z}$. Since the mapping $\Xi:\mathbb{Z}\to\mathbb{R}^{n+1+m}$ is smooth, the criterion developed in \cite[Theorem 3.2]{LiPong18} can be used to check the KL property of $\Phi$ with exponent $q\in[1/2,1)$ at $\overline{z}\in\mathcal{Z}^*$ in terms of that of $\psi$ and the surjectivity of $\mathcal{J}\Xi(\overline{z})$. The surjectivity of $\mathcal{J}\Xi(\overline{z})$ is equivalent to requiring that ${\rm Ker}\,[I\ \ \overline{\xi}\ \ \overline{V}]=\{(0_{\mathbb{R}^n},0,0_{\mathbb{R}^m})\}$, which is obviously stronger than the condition \eqref{cond-KL}. Thus, the criterion in Proposition \ref{prop-KLexp} is weaker than that of \cite[Theorem 3.2]{LiPong18}. 

 \noindent
 {\bf(b)} From the proof of Proposition \ref{prop-KLexp}, the condition \eqref{cond-KL} implies the metric regularity of $\mathcal{T}$ at $((\overline{x},\overline{s},\overline{V},\overline{L}),0_{\mathbb{R}^m})$. When the latter holds, it is highly possible for the condition \eqref{cond-KL} to hold if $\overline{\xi}$ has no zero components but $v+\overline{V}\lambda$ has a zero entry. 
\end{remark}
\section{Numerical experiments}\label{sec5}
\subsection{Implementation of Algorithm \ref{iMBA}}\label{sec5.1}
The core of Algorithm \ref{iMBA} at each iteration is to seek an inexact minimizer of subproblem \eqref{subprobkj} satisfying \eqref{inexact1}-\eqref{inexact2}. Here, we take a look at the solving of \eqref{subprobkj}. Fix any $k,j\in\mathbb{N}$. For the PD linear mapping $\mathcal{Q}_{k,j}\!:\mathbb{R}^n\to\mathbb{R}^n$ in step (2a), we choose it to be of the form $\mathcal{Q}_{k,j}=\mu_{k,j}\mathcal{I}+\!\mathcal{A}_{k,j}^*\mathcal{A}_{k,j}$, where $\mathcal{A}_{k,j}\!:\mathbb{R}^n\to\mathbb{R}^{p}$ is a linear mapping associated with $g_0$. Such a $\mathcal{Q}_{k,j}$ makes the dual of \eqref{subprobkj} have a simple composite form. Indeed, now \eqref{subprobkj} can equivalently be written as 
\begin{align}\label{EEsubprob}
	&\min_{x,s\in\mathbb{R}^n,y\in\mathbb{R}^p}\frac{1}{2}\|y\|^2+\frac{1}{2}\mu_{k,j}\|x-x^k\|^2+\langle \xi^k,x-x^k\rangle+\phi(s)+g_0(x^k)\nonumber\\
	&\qquad {\rm s.t.}\quad G(x,x^k,V^k,L^{k,j})\le 0,\ x-s=0,\ \mathcal{A}_{k,j}(x-x^k)=y.
\end{align}
The dual of \eqref{EEsubprob} (in its equivalent minimization form) has the following composite form
\begin{equation}\label{dual}	
 \!\min_{w\in\mathbb{W}}\Xi_{k,j}(w)\!:=\!\underbrace{\frac{\|\mathcal{B}_{k,j}(w)+\xi^k\|^2}{2(\mu_{k,j}\!+\!\langle\lambda,L^{k,j}\rangle)}\!-\!\langle\eta,x^k\rangle\!-\!\langle\lambda,g(x^k)\rangle\!+\!\frac{1}{2}\|\zeta\|^2\!-\!g_0(x^k)}_{\Theta_{k,j}(w)}\!+\delta_{\mathbb{R}_{+}^m}(\lambda)\!+\!\phi^*(\eta)
\end{equation}
for $w=(\lambda,\eta,\zeta)\in\mathbb{W}:=\mathbb{R}^{m}\times\mathbb{R}^{n}\times\mathbb{R}^{p}$, where $\Theta_{k,j}$ is a twice continuously differentiable convex function on the space $\mathbb{W}$, and $\mathcal{B}_{k,j}\!:\mathbb{W}\to\mathbb{R}^n$ is a linear mapping defined by
\[
\mathcal{B}_{k,j}(w):=V^k\lambda+\eta+\mathcal{A}_{k,j}^*\zeta\quad\forall w\in\mathbb{W}.
\]

The strong convexity of \eqref{subprobkj} guarantees that there is no dual gap between \eqref{subprobkj} and \eqref{dual}, and \eqref{dual} has a nonempty optimal solution set. Therefore, one can seek an inexact minimizer of \eqref{subprobkj} satisfying \eqref{inexact1}-\eqref{inexact2} by solving its dual. Considering that the gradient of $\Theta_{k,j}$ is not Lipschitz continuous, we suggest the PG method with line search (PGls) as the solver to \eqref{dual}. The iteration steps of PGls are described in Algorithm \ref{PGls} below, and its convergence certificate can be found in \cite{Kanzow22}. In the sequel, we call Algorithm \ref{iMBA} armed with Algorithm \ref{PGls} iMBA-pgls. 

\begin{algorithm}
	\renewcommand{\thealgorithm}{A}
	\caption{\label{PGls}{\bf (PGls for solving problem \eqref{dual})}}
	\textbf{Input:}\ $0<\tau_{\rm min}\le\tau_{\rm max},\,\varrho>1,\,\delta\in(0,1),w^0=(\lambda^0,\eta^0,\zeta^0)\in\mathbb{R}_{+}^m\times\mathbb{R}^n\times{\rm dom}\,\phi^*$.
	
	\noindent
	\textbf{For} $l=0,1,2,\ldots$  \textbf{do}
	\begin{enumerate}
		\item Choose $\tau_{l,0}\in[\tau_{\rm min},\tau_{\rm max}]$.
		
		\item \textbf{For} $\nu=0,1,2,\ldots$ \textbf{do}
		\begin{itemize}
			\item [(2a)] Let $\tau_{l,\nu}=\tau_{l,0}\varrho^{\nu}$. Compute an optimal solution $w^{l,\nu}$ of the problem
			\begin{equation*}
				\qquad\min_{w\in\mathbb{W}}\Theta_{k,j}(w^{l})+\langle\nabla\Theta_{k,j}(w^{l}),w-w^{l}\rangle+\frac{\tau_{l,\nu}}{2}\|w-w^{l}\|^2+\delta_{\mathbb{R}_{+}^m}(\lambda)\!+\phi^*(\eta).
			\end{equation*}
			
			\item [(2b)] If $\Xi_{k,j}(w^{l,\nu})\le \Xi_{k,j}(w^{l})-(\delta\tau_{l,\nu}/2)\|w^{l,\nu}-w^l\|^2$, go to step 3.
		\end{itemize}
		\textbf{End}
		\item Set $\nu_l:=\nu$ and $w^{l+1}:=w^{l,\nu_l}$. 
	\end{enumerate}
    \textbf{End}
\end{algorithm}

\subsubsection{Stop conditions of iMBA-pgls}\label{sec5.1.1}

According to Remark \ref{remark-alg} (c), we terminate Algorithm \ref{iMBA} at the $k$th iteration whenever $\|x^k\!-\!x^{k-1}\|\le\epsilon$ or $k>10^4$. Since Algorithm \ref{iMBA} produces at each iteration a feasible point $x^k\in\Gamma$ and a nonnegative $\lambda^k$, we also terminate its iteration once $[-\langle\lambda^k,g(x^k)\rangle]_{+}\le\epsilon_1$ and $k\ge 500$, where the tolerances $\epsilon>0$ and $\epsilon_1>0$ are specified in the experiments. 

Next we discuss how to terminate Algorithm \ref{PGls}, i.e., how to seek an inexact minimizer of \eqref{subprobkj} satisfying \eqref{inexact1}-\eqref{inexact2} by solving its dual. For each $w\in\mathbb{W}$, let $\mathcal{L}(x,s,y,w)$ be the Lagrange function of \eqref{EEsubprob} associated with $w$. Note that the minimization problem $\min_{x,s\in\mathbb{R}^n,y\in\mathbb{R}^p}\mathcal{L}(x,s,y,w)$ is strongly convex w.r.t. $x$ and $y$, and $(x^*,s^*,y^*)$ is its optimal solution if and only if 
$x^*\!=x^k-\frac{\xi^k+\mathcal{A}_{k,j}^*\zeta+\eta+V^k\lambda}{\mu_{k,j}+\langle\lambda,L^{k,j}\rangle}, \eta\in\partial\phi(s^*)$ and $y^*=\zeta$. It is not hard to verify that if $\eta^{+}={\rm prox}_{\tau^{-1}\phi^*}(\eta-\tau^{-1}\nabla_{\!\eta}\Theta_{k,j}(w))$ for a step size $\tau>0$, then $\eta^+\in\partial\phi(\widetilde{s})$ with $\widetilde{s}=\tau(\eta-\eta^{+})-\nabla_{\!\eta}\Theta_{k,j}(w)=\tau(\eta-\eta^{+})+x^*$. Apparently, when $\eta^{+}$ is sufficiently close to $\eta$, such $\widetilde{s}$ is actually a small perturbation of $x^*$, and when $w$ is an optimal solution of \eqref{dual}, the associated $(x^*,\widetilde{s}^*,y^*)$ is optimal to \eqref{EEsubprob}. When applying the PGls to solve \eqref{dual}, the generated iterate sequence $\{\eta^l\}_{l\in\mathbb{N}}$ does satisfy such a relation, i.e., $\eta^{l+1}\!={\rm prox}_{\tau_l^{-1}\phi^*}(\eta^l\!-\!\tau_l^{-1}\nabla_{\!\eta}\Theta_{k,j}(w^l))$. By this, we set $x_{k,j}^{l}:=x^k-\frac{\xi^k+\mathcal{A}_{k,j}^*\zeta^l+\eta^l+V^k\lambda^l}{\mu_{k,j}+\langle\lambda^l,L^{k,j}\rangle}$ at  the current iterate $w^l=(\lambda^l,\eta^l,\zeta^l)$,  and check if $(x_{k,j}^{l},\eta^l,\lambda^l)$ satisfies \eqref{inexact1}-\eqref{inexact2}. When 
\[
 F_{k,j}(x_{k,j}^{l})\le F_{k,j}(x^k), C_{k,j}(x_{k,j}^{l},\lambda^{l})\le\frac{\beta_C}{2}\|x_{k,j}^{l}-x^k\|^2,S_{k,j}(x_{k,j}^{l},\eta^{l},\lambda^{l})\le \beta_S\|x_{k,j}^l-x^k\|
 \]
hold, we terminate Algorithm \ref{PGls} at the iterate $w^{l}$, and $x_{k,j}^{l}$ serves as the inexact minimizer $y^{k,j}$ of \eqref{subprobkj}, which along with $(\eta^{l},\lambda^{l})$ satisfies conditions \eqref{inexact1}-\eqref{inexact2}. 
\subsubsection{Setup of parameters}\label{sec5.1.2}

As mentioned in Section \ref{sec5.1}, we take $\mathcal{Q}_{k,j}=\mu_{k,j}\mathcal{I}+\mathcal{A}_{k,j}^*\mathcal{A}_{k,j}$. The choice of the linear operator $\mathcal{A}_{k,j}$ depends on the second-order information of $f_0$, the smooth part of $g_0$; see the subsequent experiments. Since $\mathcal{A}_{k,j}^*\mathcal{A}_{k,j}$ is positive semidefinite, the choice of $\mu_{k,j}$ has a great influence on the approximation effect of $F_{k,j}$ to $F$. Similarly, the choice of $L^{k,j}$ determines the approximation effect of $\Gamma_{k,j}$ to $\Gamma$. Our preliminary tests show that the updating rules of $L^{k,j}$ and $\mu_{k,j}$ in steps (2c)-(2d) work well with $\tau=2$, as long as the initial $\mu_{0,0}$ and $L^{0,0}$ are appropriately chosen. Apparently, the larger $\mu_{0,0}$ and $L^{0,0}$ make the first subproblem (so the subsequent ones) have a big gap from the original problem \eqref{prob} though its computation is relatively easy, and now iMBA-pgls will return a bad stationary point. In view of this, we choose $\mu_{0,0}=\,{\rm lip}\,\nabla\!f_0(x^0)$ and $L^{0,0}=0.05\,\ell_{\nabla\!f}(x^0)$, where $f=(f_1,\ldots,f_m)^{\top}$ is the smooth part of the constraint function $g$, and ${\rm lip}\,\nabla\!f_0(x^0)$ and $\ell_{\nabla\!f}(x^0)$ can be estimated by the Barzilai-Borwein rule \cite{Barzilai88}. We also find that Algorithm \ref{iMBA} is robust to $\alpha\in[10^{-12},10^{-4}]$, so choose $\alpha=10^{-6}$ for the tests. To sum up, for the subsequent tests, the parameters of Algorithm \ref{iMBA} are chosen as follows: 
\begin{align*} &\mu_{\min}=10^{-16},\,\mu_{\max}=10^{16},\,L_{\min}=10^{-16},\,L_{\max}=10^{16},\\
&\beta_{C}\!=\!10^{10},\,\beta_{S}\!=\!10^6,\,\alpha=10^{-6},\,\tau=2.
\end{align*}

Next we look at the choice of the parameters in Algorithm \ref{PGls}. Fix any $k,j\in\mathbb{N}$. Our preliminary tests indicate that Algorithm \ref{PGls} is very robust to the parameter $\delta$ whenever $\delta\le 10^{-2}$. However, the initial $\tau_{0,0}$ or the step-size $\tau_{l,\nu}^{-1}$ has a great impact on the efficiency of Algorithm \ref{PGls}. Apparently, a smaller initial $\tau_{0,0}$ will lead to a larger step-size. Based on this, we choose $\tau_{0,0}=10^{-8}\|V^k\|^2$. The following parameters
\[  \delta=10^{-6},\,\varrho=10,\,\tau_{0,0}=10^{-8}\|V^k\|^2\ \ {\rm and}\ \ l_{\rm max}=2000
\]
are used for the subsequent tests, where $l_{\rm max}$ is the maximum number of iterations. 
\subsection{Test problems}\label{sec5.2}
We shall test the performance of iMBA-pgls for solving quadratic DC constrained (QDCC) problems. The functions $g_0$ and $\phi$ in the objective of QDCC problems are \begin{equation}\label{obj}
 g_0(x):=f_0(x)-h_0(x)\ {\rm with}\ h_0(x):=0.01\|x\|\ \ {\rm and}\ \ \phi(x)=0.01\|x\|_1.
\end{equation}
We consider two types of $f_0$ for the subsequent tests. One type is the quadratic function 
\begin{equation}\label{f0-type1}
 f_0(x):=\|Y_0x\|^2+2\omega_{0}(b_0/\|b_0\|)^{\top}x
\end{equation}
where, $\omega_0$ is a constant, $Y_0\in\mathbb{R}^{p\times n}$ with $p=\lfloor n/2\rfloor$ is a matrix, and $b_0\in\mathbb{R}^n$ is a vector. The other one is the nonconvex loss function introduced in \cite{Aravkin12} to deal with the data contaminated by heavy-tailed Student-$t$ noise, which have the following form
\begin{equation}\label{f0-type2}
 f_0(x):=\theta(Ax-b)\quad{\rm with}\  \ \theta(u):=\textstyle{\sum_{i=1}^N}\log[1\!+\!4u_i^2]. 
\end{equation}
The data $A\in\mathbb{R}^{N\times n}$ and $b\in\mathbb{R}^{N}$ are randomly generated in the same way as in \cite{Milzarek14}.

The constraint functions $g_i$ for $i\in[m]$ of QDCC problems take the DC functions 
\begin{equation}\label{mapg}
 g_i(x):=x^{\top}Q_ix-x^{\top}P_ix+2\langle b_i,x\rangle +c_i\ \ {\rm with}\ P_i=10^5I,
\end{equation}
where $Q_1,\ldots,Q_m$ are the $n\times n$ PD matrices, $b_1,\ldots,b_m\in\mathbb{R}^n$, and $c_1,\ldots,c_m\in\mathbb{R}$. We follow the same way as in \cite{Auslender10} to generate the $m$ PD matrices $Q_i$. That is, every $Q_i$ is of the form $Q_i=Y_iD_iY_i$, where $D_i$ is an $n\times n$ diagonal matrix with diagonal elements shuffled randomly from the set $\{10^{\frac{10(j-1)}{n-1}}\,|\,j\in[n]\}$, and $Y_i=I-2\frac{y_iy_i^{\top}}{\|y_i\|^2}$ is an $n\times n$ random Householder orthogonal matrix with $I$ being the $n\times n$ identity matrix, and the components of $y_i\in\mathbb{R}^n$ chosen randomly from $(-1,1)$. Clearly, the spectral norm $\|Q_i\|$ of every $Q_i$ equals $10^{10}$. To generate the data $b_i$ and $c_i$ for $i\in[m]$ such that the feasible set $\Gamma=g^{-1}(\mathbb{R}_{-}^m)$ is nonempty, we consider an equivalent reformulation of $g(x)\in\mathbb{R}_{-}^m$. By the expression of $Q_i$, we have $Q_i=B_i^{\top}B_i$ with $B_i=D_i^{1/2}Y_i$ being nonsingular. Then, with $h_i=(B_i^{-1})^{\top}b_i$ and $d_i^2=\|h_i\|^2-c_i$, the constraint functions $g_i$ are reformulated as 
\[
g_i(x)=\|B_ix+h_i\|^2-x^{\top}P_ix-d_i^2\le 0\quad{\rm for}\ i\in[m].
\]

Let $x^0\in\mathbb{R}^n$ be a vector generated randomly. To ensure that $x^0\in\Gamma$, for each $i\in[m]$, we randomly choose the components of $h_i$ from $(-1,1)$ and $s_i$ from $[0,1)$, and then set $d_i^2=\|B_ix^0+h_i\|^2-x^{\top}P_ix+s_i$. Clearly, $\|B_ix^0+h_i\|^2-x^{\top}P_ix-d_i^2\le 0$ due to the nonnegativity of $s_i$, and consequently, $x^0\in\Gamma$. For each $i\in[m]$, after generating $h_i$ and $d_i$, we set $b_i=B_i^{\top}h_i$ and $c_i=\|h_i\|^2\!-d_i^2$. 
\subsection{Validation of theoretical results}\label{sec5.3}

By Lemma \ref{lemma-welldef} (ii), when $x^k$ is a non-stationary point of \eqref{prob}, the inner loop of Algorithm \ref{iMBA} necessarily stops within a finite number of steps. Figure \ref{fig1} plots the number of steps required by the inner loop at each iteration. We see that the number of steps for the inner loop is basically not more than $\textbf{3}$, and attains $\textbf{3}$ only at several iterations of the outer loop. This validates the conclusion of Lemma \ref{lemma-welldef} (ii). 
\begin{figure}[h]
\centering
\subfigure[\label{fig1a}]{\includegraphics[scale=0.43]{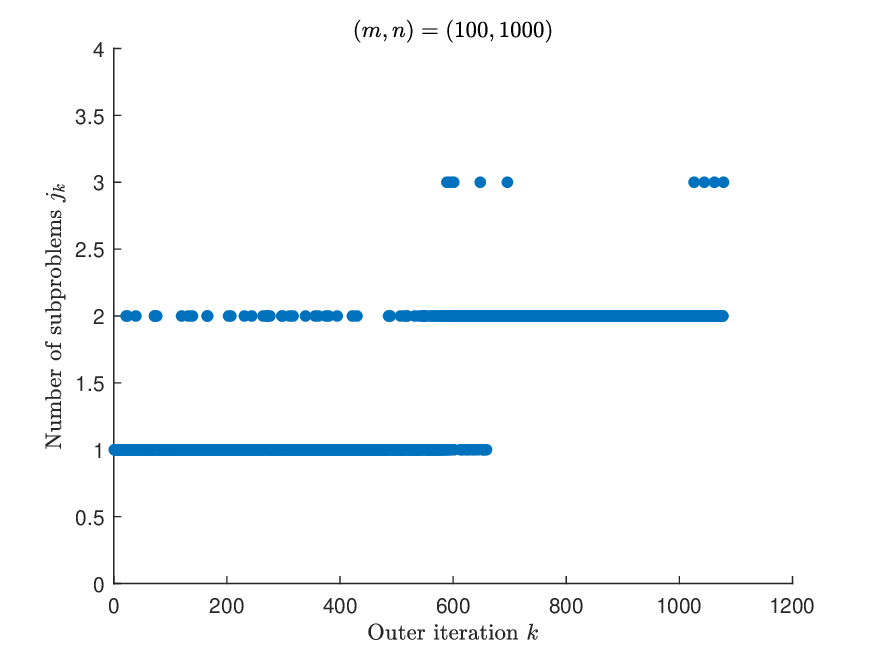}}
\subfigure[\label{fig1b}]{\includegraphics[scale=0.43]{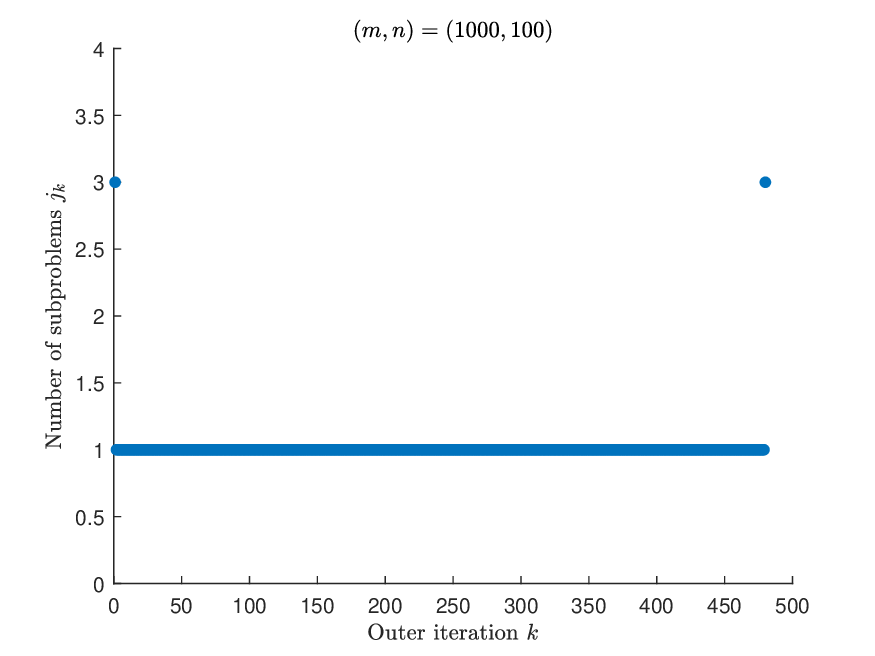}} 
\setlength{\abovecaptionskip}{2pt}
\setlength{\belowcaptionskip}{2pt}
\caption{Performance of the inner loop of Algorithm \ref{iMBA} for $f_0$ from \eqref{f0-type1} with $\omega_0=10^3$}
\label{fig1}
\end{figure}

Figure \ref{fig_rate} plots the convergence behavior of the iterate sequence of Algorithm \ref{iMBA}. By Theorem \ref{globalconv}, the iterate sequence is convergent under Assumptions \ref{ass2}-\ref{ass4} and the KL property of $\Phi_{\widetilde{c}}$. The function $\Phi_{\widetilde{c}}$ corresponding to $F$ with $g_0,\phi$ and $g$ from Section \ref{sec5.2} is semialgebraic, so is a KL function of exponent $q$ for some $q\in(0,1)$. Notice that Assumption \ref{ass3} does not necessarily hold. While Assumption \ref{ass4} automatically holds if the constraint $g(x)\in\mathbb{R}_{-}^m$ satisfies the MFCQ at every feasible point. Since it is highly possible for the latter to hold when $n\approx m$ or $n\gg m$, the iterate sequences of Algorithm \ref{iMBA} for those problems with $n\approx m$ or $n\gg m$ will have better convergence. The curves of Figure \ref{fig_rate} indeed validate this result, and also demonstrate the linear convergence of the iterate sequences. Note that  $\Phi(z)=\psi(x,T_0(x,s,\xi),T(x,s,V,L))$ with $\psi$ being a piecewise linear-quadratic function for this group of test instances. Such $\psi$ satisfies the requirement of Proposition \ref{prop-KLexp}. Thus, the linear convergence behavior of the curves in Figure \ref{fig_rate} means that the condition \eqref{cond-KL} also holds for  this group of test instances.     
\begin{figure}[h]
\centering
\subfigure[\label{fig1a}]{\includegraphics[scale=0.44]{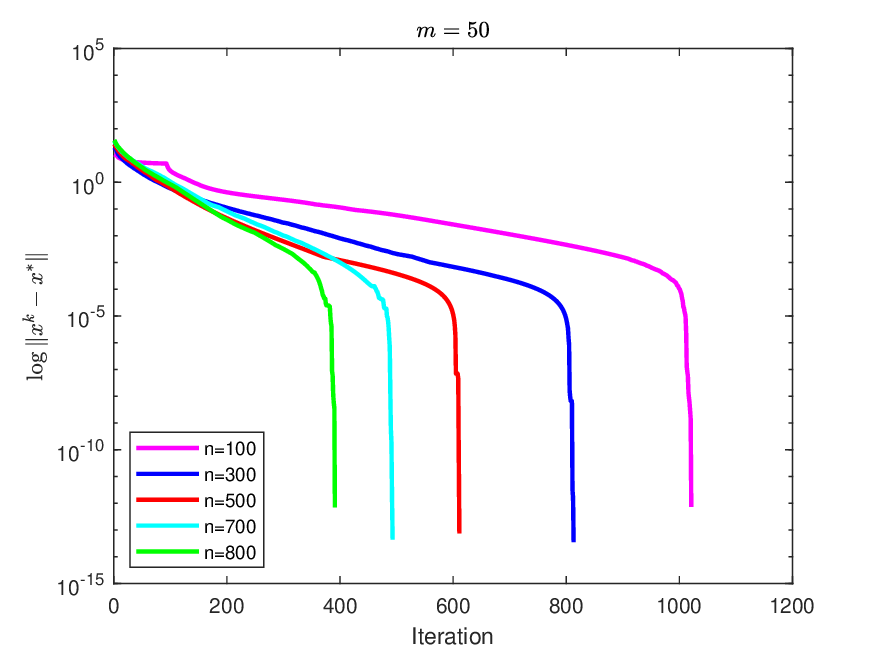}}
\subfigure[\label{fig1b}]{\includegraphics[scale=0.43]{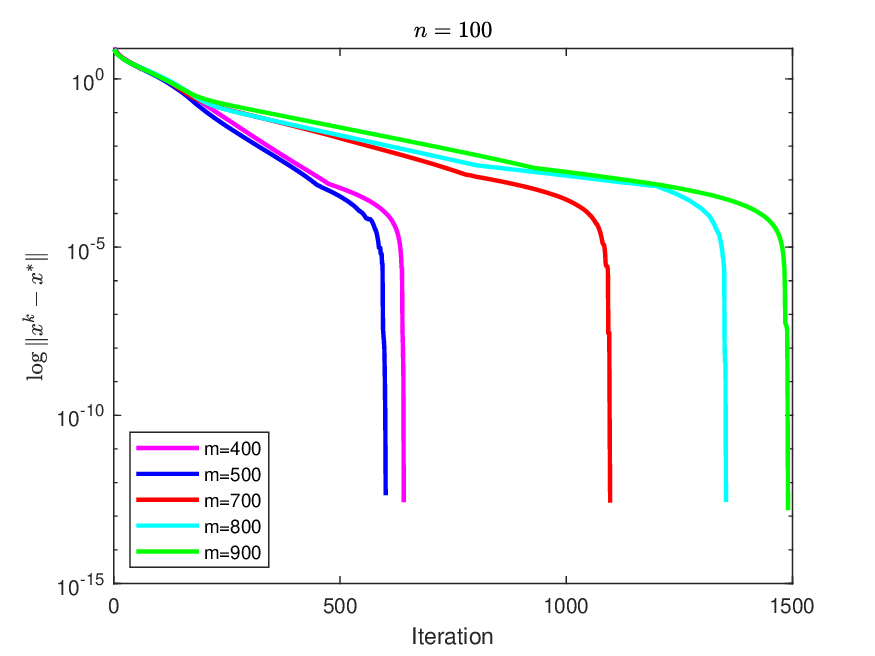}}
\setlength{\abovecaptionskip}{2pt}
\setlength{\belowcaptionskip}{2pt}
\caption{Convergence behavior of the iterate sequence for $f_0$ from \eqref{f0-type1} with $\omega^0=10^4$}
\label{fig_rate}
\end{figure}
\subsection{Numerical comparisons for QDCC problems}\label{sec5.4}

\subsubsection{$\ell_1\!-\!\ell_2$ regularized quadratic objective functions}\label{sec5.4.1}

We test the performance of iMBA-pgls for solving QDCC problems with $\ell_1-\ell_2$ regularized quadratic objective functions. The DC regularizer aims at testing the efficiency of iMBA-pgls rather than seeking a sparse solution. For each $k,j\in\mathbb{N}$, we choose the linear mapping $\mathcal{A}_{k,j}\!:\mathbb{R}^n\to\mathbb{R}^p$ to be $\mathcal{A}_{k,j}x=Y_0x$ for $x\in\mathbb{R}^n$, so the corresponding linear operator $\mathcal{Q}_{k,j}$ has the expression $\mathcal{Q}_{k,j}x=\mu_{k,j}x+Y_0^{\top}Y_0x$ for $x\in\mathbb{R}^n$. 

We compare the performance of iMBA-pgls with that of Algorithm \ref{DCA} above, a DCA for solving the QDCC problems in Section \ref{sec5.2}. For Algorithm \ref{DCA}, there is no full convergence guarantee for its iterate sequence, and we use it just for numerical comparisons. We apply MOSEK to solve the quadratic conic reformulations of the subproblems of Algorithm \ref{DCA}, and now Algorithm \ref{DCA} is called DCA-MOSEK. During the tests, we adopt the default setting of MOSEK and terminate DCA-MOSEK at the $k$th iteration once $\|x^k-x^{k-1}\|\le\epsilon$. For a fair comparison, the two solvers starts from the same feasible point $x^0$.
\begin{algorithm}
\renewcommand{\thealgorithm}{2}
\caption{\label{DCA}{\bf (DCA for solving problem \eqref{prob})}}
\textbf{Input:}\ $x^0\in\Gamma$.
	
\noindent
\textbf{For} {$k=0,1,2,\ldots$}  \textbf{do}
	
	Choose $\xi^k\in\partial h_0(x^k)$;
	
	Seek an optimal solution $x^{k+1}$ of the convex quadratic program
	\begin{align*}
		&\min_{x\in\mathbb{R}^n} f_0(x)+\phi(x)-\langle\xi^k,x\rangle-h_0(x^k)\\
		&\ \ {\rm s.t.}\ x^{\top}Q_ix-2\langle P_ix^k,x-x^k\rangle+2b_i^{\top}x+c_i-(x^k)^{\top}P_ix^k\le 0\quad\forall i\in[m].
	\end{align*}
	\textbf{End}
\end{algorithm}

Table \ref{L12-QCQPs} reports the results of iMBA-pgls and DCA-MOSEK with $\epsilon=10^{-5}$ and $\epsilon_1=10^{-7}$. Since MOSEK solves the quadratic conic reformulations of the subproblems with an interior point solver, DCA-MOSEK starting from $x^0\in\Gamma$ always returns a feasible solution to this class of problems. We see that for this class of QDCC problems, DCA-MOSEK requires less iterations for the test instances with $\omega_0=10^4$ but more running time when $n$ is fixed. iMBA-pgls produces better objective values than DCA-MOSEK for all test instances, and requires much less running time than the latter for $n\ge 1000$ when $m$ is fixed, and for $m\ge 1000$ when $n$ is fixed. This demonstrates the advantage of iMBA-pgls for solving this class of QDCC problems with a larger $m$ or $n$.    
\begin{table}[h]
\small
\centering
\setlength{\tabcolsep}{4.0pt}
\setlength{\belowcaptionskip}{2pt}
\caption{\small Numerical results for $\ell_1\!-\!\ell_2$ regularized quadratic cost functions}
\label{L12-QCQPs}
\begin{tabular}{@{\extracolsep{\fill}}cccccccccc@{\extracolsep{\fill}}}
\hline
\multicolumn{1}{c}{}&\multicolumn{1}{c}{}& \multicolumn{1}{c}{}  & \multicolumn{4}{c}{iMBA-pgls} & \multicolumn{3}{c}{DCA-MOSEK} \\
		\hline
		$\omega_0$ &$n$ & $m$ &  iter & Fval &  time(s) & compl &  iter & Fval & time(s) \\
		\hline
		\multirow{10}{*}{$10^4$} & 100 &\multirow{5}{*}{100} 
		& 550.0 & 	  {\bf-1.5763e+5} &	   6.34  &	  1.0389e-5
		& 3.0  	&  -1.5758e+5 	  & 3.61\\ 
		& 200 &\multirow{5}{*}{} 
		& 526.0 & 	  {\bf-2.3142e+5} &	   5.61  &	  1.4801e-5 
		& 6.0  	&  -2.3055e+5 	  & 15.38\\
		& 500 & \multirow{5}{*}{} 
		& 325.0 & 	  {\bf-3.6575e+5} &	   5.63  &	  5.6332e-4 
		& 3.0  	&  -3.5813e+5 	  & 63.71\\
		& 1000 &\multirow{5}{*}{} 
		& 478.0 & 	  {\bf-4.8139e+5} &	  228.42  &	  5.0061e-6  
		& - & -&-\\
		& 2000 &\multirow{5}{*}{} 
		& 405.0  &	  {\bf-6.5331e+5} &	  106.41  &	  1.7908e-6
		& - & - & - \\ 
		\cline{2-10}
		& \multirow{5}{*}{100} & 200 
		& 558.0 & 	  {\bf-1.4415e+5} &	   7.88  &	  6.2327e-6
		& 3.0  	&  -1.4410e+5 	  & 7.90\\
		& \multirow{5}{*}{} &500 
		& 1095.0 & 	  {\bf-1.3054e+5} &	  39.26  &	  9.6592e-6
		& 3.0  	 & -1.3049e+5 	  & 30.81\\
		& \multirow{5}{*}{} &1000 
		& 1130.0  &	  {\bf-1.2770e+5} &	  55.17  &	  1.2426e-5
		& 4.0  	  &-1.2765e+5 	  & 106.89\\
		& \multirow{5}{*}{} &2000 
		& 1003.0  &	  {\bf-1.2441e+5} &	  108.43  &	  7.1233e-6
		& 4.0  	  &-1.2436e+5 	  & 351.72\\ 
		& \multirow{5}{*}{} &3000 
		& 930.0  &	  {\bf-1.2340e+5} &	  218.66  &	  6.6391e-6
		& 4.0  	 & -1.2336e+5 	  & 737.58\\
		\hline
		\multirow{10}{*}{10} & 100 &\multirow{5}{*}{100} 
		& 1275.0 & 	  {\bf -1.0960e+2} 	&   2.89  &	  9.0792e-8
		& 4.0  	 & -1.0837e+2 	&   4.40\\
		& 200 &\multirow{5}{*}{ } 
		& 3068.0 & 	  {\bf -1.6670e+2} 	&  13.94  &	  9.7689e-8
		& 4.0  	 & -1.6627e+2 	&  11.21\\
		& 500 &\multirow{5}{*}{ } 
		& 9153.0 & 	  {\bf-2.8080e+2} 	&  127.49  &	  9.7682e-8
		& 4.0  	 & -2.8078e+2 	&  80.87\\
		& 1000 &\multirow{5}{*}{ } 
		& 10000.0 &	  {\bf-3.9194e+2} 	&  491.83  	&  8.9557e-6
		& - & - & -\\
		& 2000 &\multirow{5}{*}{ } 
		& 500.0  &	  {\bf -2.5515e+2} 	& 134.51  &  0
		& - & - & -\\
		\cline{2-10}
		& \multirow{5}{*}{100} & 200 
		& 1255.0 & 	  {\bf-1.0171e+2} &	   6.02  &	  7.6043e-8
		& 3.0  	 & -1.0052e+2 	  & 8.37\\
		& \multirow{5}{*}{} & 500 
		& 1272.0 & 	  {\bf-8.6065e+1} &	   8.05  &	  9.0931e-8
		& 4.0  	 & -8.4153e+1 	  & 35.14\\
		& \multirow{5}{*}{} & 1000 
		& 1130.0 & 	  {\bf-8.5009e+1} &	  13.32  &	  9.4000e-8
		& 4.0  	 & -8.3492e+1 	  & 107.11\\
		& \multirow{5}{*}{} & 2000
		& 1693.0 & 	  {\bf -7.8086e+1} &  44.49  &	  8.7478e-8
		& 4.0  	 & -7.6325e+1 	  & 356.70\\
		& \multirow{5}{*}{} & 3000 
		& 1699.0 & 	  {\bf-7.7462e+1} &  52.38  &	  5.4841e-8
		& 6.0  	 & -7.5910e+1 	  & 805.80\\
		\hline
	\end{tabular}
\end{table}
\subsubsection{$\ell_1\!-\!\ell_2$ regularized Student's $t$-regression cost functions}

We test the performance of iMBA-pgls for solving QDCC problems with $\ell_1-\ell_2$ regularized Student's $t$-regression objective functions. Note that $\nabla^2f_0(x)=A^{\top}\nabla^2\theta(Ax-b)A$ for $x\in\mathbb{R}^n$. When testing this group of instances, for each $k,j\in\mathbb{N}$, we take $\mathcal{A}_{k,j}={\rm diag}([\omega^k]_{+}^{1/2})A$, where $\omega^k\in\mathbb{R}^N$ is the diagonal vector of the diagonal matrix $\nabla^2\theta(Ax^k\!-\!b)$. Since MOSEK can be employed to solve the quadratic conic reformulation of \eqref{subprobkj}, we compare the performance of iMBA-pgls with that of Algorithm \ref{iMBA} armed with MOSEK to solve its subproblems (iMBA-MOSEK). When applying MOSEK to solve the quadratic conic reformulation of \eqref{subprobkj}, we set the stop tolerances for the dual feasibility, primal feasibility and relative gap of its interior-point solver to be $10^{-5}$, and the other parameters are all set to be the default ones. For the fairness of comparison, iMBA-MOSEK starts from the same feasible point $x^0$, and stops under the same condition as for iMBA-pgls. 

Table \ref{table_dct} reports the results of iMBA-pgls and iMBA-MOSEK with $\epsilon=10^{-5}$ and $\epsilon_1=10^{-7}$. Observe that iMBA-MOSEK returns better objective values for more test examples. This indicates that, when $f_0$ is highly nonlinear, a powerful solver to the subproblems of Algorithm \ref{iMBA} will improve its performance. Unfortunately, the running time of iMBA-MOSEK is much more than that of iMBA-pgls, and the former is at least \textbf{7} times the latter when $m$ is fixed to be $100$. Thus, when $f_0$ is highly nonlinear, seeking a more robust and faster solver than PGls to the subproblems is still requisite.   
\begin{table}[h]
\small
\centering
\setlength{\tabcolsep}{2.0pt}
\setlength{\belowcaptionskip}{2.0pt}
\caption{\small Numerical results for $\ell_1\!-\!\ell_2$ regularized Student's $t$-regressions cost functions}
\label{table_dct}
\begin{tabular}{@{\extracolsep{\fill}}cccccccccc@{\extracolsep{\fill}}}
\hline
\multicolumn{1}{c}{}& \multicolumn{1}{c}{}  & \multicolumn{4}{c}{iMBA-pgls} & \multicolumn{4}{c}{iMBA-MOSEK} \\
\hline
 $n$ & $m$ &  iter & Fval &  time(s) & compl &  iter & Fval & time(s) & compl\\
		\hline
		  300 &\multirow{4}{*}{50} 
		& 168  &	  4.6455e+2 	&  6.76  &	  4.3972e-10
		& 500  &	  {\bf4.6409e+2} 	&  30.48  &	  3.9712e-9\\
		  500 &\multirow{4}{*}{}
		& 960  &	  {\bf6.8314e+2} 	&  9.73  &	  3.3109e-8
		& 500  &	  6.8379e+2 	&  58.64  &	  3.8367e-8\\
		800 &\multirow{4}{*}{}
		& 519  &	  1.0528e+3 	&  9.66  &	  1.9724e-8
		& 3000 &	  {\bf1.0522e+3} 	&  597.96 &	  9.4486e-7\\
		1000 &\multirow{4}{*}{}
		& 818  &	  1.3984e+3 	&  21.31  &	  9.4625e-8
		& 3000 &	  {\bf1.3977e+3} 	&  769.32 &	  3.3973e-7\\ \cline{1-10}
		  300 &\multirow{4}{*}{100} 
		& 121  &	  4.6471e+2 	&  6.29  &	  1.2036e-10
		& 500  &	  {\bf4.6420e+2} 	&  59.00 & 	  1.5984e-9	\\ 
		500 &\multirow{4}{*}{}
		& 822  &	  6.8489e+2 	&  12.80  &	  7.2792e-8
		& 500  &	  {\bf6.8427e+2} 	&  734.79  &  6.9279e-9	\\
		  800 &\multirow{4}{*}{}
		& 556  &	  1.0580e+3 	&  17.42  &	  1.8459e-8
		& 530  &	  {\bf1.0575e+3} 	&  1198.25 & 	  9.8133e-8	\\
		  1000 &\multirow{4}{*}{}
		& 547  &	  1.4043e+3 	&  26.17  &	  6.3158e-10
		& 894  &	  {\bf1.4036e+3} 	&  2689.58 & 	  9.9363e-8\\
		\hline
        \multirow{3}{*}{800} & 200 & 688 &  1.0589e+3 & 44.36 & 3.6767e-8 & - & - & - & -\\
        \multirow{3}{*}{}    & 400 & 625 &  1.0626e+3 & 76.41 & 9.2748e-8 & - & - & - & -\\
        \multirow{3}{*}{}    & 600 & 668 &  1.0630e+3 & 120.98 & 5.8394e-8 & - & - & - & -\\
        \hline
	\end{tabular}
\end{table}
\section{Conclusion}\label{sec6.0}

We proposed an inexact MBA method with a practical inexactness criterion for the constrained upper-$\mathcal{C}^2$ problem \eqref{prob}, thereby solving the dilemma that the existing MBA methods focus on the constrained case with the $\mathcal{C}^{1+}$ constraint functions. For the proposed inexact MBA method, we conducted the convergence analysis under the MSCQ and the BMP, and established the full convergence of the iterate sequence under the KL property of $\Phi_{\widetilde{c}}$, and the convergence rate on the iterate and objective value sequences under the KL property of $\Phi$ with exponent $q\in[1/2,1)$. As far as we know, this is the first full convergence result of the MBA methods on the iterate sequence without requiring the MFCQ. Among others, the error bound in Proposition \ref{prop-eboundk} and the construction of $\Phi_{\widetilde{c}}$ play a crucial role. In addition, a checkable condition, weaker than the one in \cite[Theorem 3.2]{LiPong18}, was provided to identify whether $\Phi$ satisfies the KL property of exponent $q\in[1/2,1)$ on the given critical point. 

The implementation of the inexact MBA method was discussed in Section \ref{sec5}. The archive version of this paper includes the numerical comparisons with the software MOSEK on several classes of test problems, which demonstrate the superiority of iMBA-pgls in the quality of solutions and running time.

\section*{Acknowledgments}
The authors thank Prof. Mordukhovich from Wayne State University and Prof. Jin Zhang from Southern University of Science and Technology for their helpful comment on the property of upper-$\mathcal{C}^2$ functions in Lemma \ref{lemma-major} (ii).

\bibliographystyle{siam}
\bibliography{references}

\end{document}